\documentclass[11pt,a4paper]{article}
\usepackage{amsmath,amsxtra,amsthm,amssymb,makeidx,graphics,mathrsfs}
\usepackage[margin=2.5cm]{geometry}
\usepackage{enumerate}
\usepackage{stmaryrd}
\usepackage{pdflscape}
\usepackage{indentfirst}
\usepackage{capt-of}
\usepackage{tikz}
\usepackage{dsfont}
\usetikzlibrary{matrix}
\usepackage{tikz-3dplot}
\usepackage{setspace}
\usepackage{caption}
\usepackage[font+=large]{subcaption}
\usepackage[all]{xy}
\usepackage{bbm}
\usepackage{slashbox}
\usepackage{graphicx}
\usepackage{xcolor}

\newtheorem{theorem}{Theorem}[section]
\newtheorem{proposition}[theorem]{Proposition}
\newtheorem{lemma}[theorem]{Lemma}
\newtheorem{corollary}[theorem]{Corollary}
\newtheorem{example}[theorem]{Example}
\newtheorem{definition}[theorem]{Definition}

\newtheorem*{notation}{Notation}
\newtheorem*{theorem*}{Theorem}

\theoremstyle{remark}
\newtheorem*{remark}{Remark}

\DeclareMathAlphabet{\mathpzc}{OT1}{pzc}{m}{it}
\DeclareMathOperator{\tensor}{\otimes}
\DeclareMathOperator{\isom}{\cong}
\DeclareMathOperator{\I}{\mathcal{I}}
\DeclareMathOperator{\F}{\mathbb{F}}
\DeclareMathOperator{\FA}{\mathbb{F}\mathit{A}}
\DeclareMathOperator{\FG}{\mathbb{F}\mathit{G}}
\DeclareMathOperator{\FH}{\mathbb{F}\mathit{H}}

\DeclareMathOperator{\T}{\mathbbmss{T}}
\DeclareMathOperator{\Z}{\mathbb{Z}}
\DeclareMathOperator{\B}{\mathbbm{B}}
\DeclareMathOperator{\sB}{\mathbbm{b}}
\DeclareMathOperator{\C}{\mathscr{C}}
\DeclareMathOperator{\D}{\mathscr{D}}

\DeclareMathOperator{\A}{\mathbbm{A}}

\DeclareMathOperator{\Hom}{\mathrm{Hom}}
\DeclareMathOperator{\PHom}{\mathrm{\underline{Hom}}}
\DeclareMathOperator{\End}{\mathrm{End}}
\DeclareMathOperator{\Ext}{\mathrm{Ext}}
\DeclareMathOperator{\PMod}{\hyp\mathrm{\underline{Mod}}}
\DeclareMathOperator{\Pmod}{\hyp\mathrm{\underline{mod}}}

\DeclareMathOperator{\Mod}{\hyp\mathrm{mod}}

\DeclareMathOperator{\Ric}{\mathbf{Ric}}

\renewcommand{\S}{\mathbbmss{S}}

\mathchardef\hyp="2D

\makeatletter
\renewenvironment{proof}[1][\proofname]{\par
    \pushQED{\qed}%
    \normalfont \topsep0\p@\@plus6\p@ \labelsep1em\relax
    \trivlist
    \item[\hskip\labelsep\indent\bfseries #1]\ignorespaces
    \mbox{}
}{
    \popQED\endtrivlist\@endpefalse
}
\makeatother

\begin{document}

\author{William Wong}
\title{Some perverse equivalences of $SL(2,q)$ in its defining characteristic}
\maketitle

\subsection*{Abstract}
\addcontentsline{toc}{chapter}{Abstract}
In this article, we study the modular representations of the special linear group of degree two over a finite field in defining characteristic. In particular, we study the automorphisms of derived category of representations. We have been able to obtain a new type of autoequivalence. 

This autoequivalence has some uncommon features. It is more conveniently conceived and proved using the representation theory of its Brauer correspondence but at the same time it can be very neatly described, using a type of derived equivalence called perverse equivalence, in global settings.

\section{Introduction}

The relation between blocks of group algebra and its subgroups is a major topic in representation theory. Let $G$ be a group and $\F$ be an algebraically closed field, let $B$ be a block of the group algebra $\FG$. Brauer's main theorems suggest that such a block with defect group $D$ can be associated with a block $b$ in $\F N_G(D)$ with the same defect group. The two blocks $B$ and $b$ are said to be Brauer correspondent. In particular Brou\'e's abelian defect conjecture \cite[6.2, Question]{Broue}\cite[p.132]{Koenig} has been studied a lot by the research community, the most successful case being that of a cyclic defect group \cite{Rouquier}. However in general the abelian case is yet to be solved, and most of the work is a case-by-case analysis.

Here we focus on the special linear group of degree two over a finite field in defining characteristic. The Brou\'e conjecture in this case is solved by Okuyama \cite{Okuyama} and Yoshii \cite{Yoshii} by a slightly unconventional manner. They use an implicit generating method which is different from the work of other authors: They construct a chain of derived equivalent algebras and they show that each end is Morita equivalent to the required blocks respectively.

In this article a similar method to Okuyama is used to prove directly that there is an block-exchanging autoequivalence in the derived category of the summands of full defect blocks of $SL_2(q)$ in defining characteristic (Theorem \ref{main}). We do this by considering their local module category, in which there is a self-equivalence in the module category induced by a functor. Using this functor we can deduce some extension relations in the global module category. Then this information can be arranged in a relatively nice manner.

To utilise this information we have to know the notion of perverse equivalence, a recent tool developed by Chuang and Rouquier in a preprint paper \cite{ChuangRouquier}. It considers a subset of derived equivalences, which is inspired from the gluing of perverse sheaves, and can be described using combinatorial data. These perverse equivalences are far easier to handle compared to general derived equivalences. In some cases one can directly observe such equivalences by looking at the correspondence of simple objects (Example \ref{elementary}). Dreyfus-Schmidt \cite{Leo} has taken the equivalence between Serre subcategories and give the notion of poset perverse equivalences, which will be particularly useful when considering composition of two perverse equivalences (which is not guaranteed to be perverse).

Let $p$ be a prime number, $q=p^n$ a positive power of $p$, $G=SL_2(q)$ the special linear group of degree two over the field of $q=p^n$ elements. Let $\F$ be an algebraically closed field of characteristic $p$. The autoequivalence is on the full defect block(s) of the group algebra $\FG$. We shall introduce its representations (as global modules). Let $P$ be a Sylow $p$-subgroup and $H=N_G(P)$ the normaliser of $P$ in $G$. We shall also introduce the representations of $\FH$ as local modules. Global and local modules are closely related via Brauer correspondence of blocks, Green correspondence of non-projective modules and equivalence between their stable module category. 

Base-$p$ numbers are very useful throughout the article. For this article we will use the base-$p$ representation of integers as defined below:
\begin{definition}\label{basep}
	For any integer $a$, $0 \leq a \leq q-1$, define $a_0, a_1,...,a_{n-1}$ to be $n-1$ integers between 0 and $p-1$ inclusive such that \begin{equation*} a=\sum_{i=0}^{n-1}p^i a_i \end{equation*} is the base-$p$ expression of $a$. On the other hand, an $n$-tuple of base-$p$ digits $(a_0,a_1,...,a_{n-1})$ uniquely determine an integer $a$ between 0 and $q-1$.
\end{definition}
  
We use the position of the algebra to indicate its side of action on the module. For example an $A$-module has $A$-action from the left. A module-$B$ is a right $B$-module in the usual sense. An $A$-bimodule-$B$ means there is an left $A$-action and right $B$-action on the module.
In this article we compose functions and functors from right to left. We always assume all vector spaces are finite dimensional, direct sums are finite. We adopt the following convention for subscripts surrounding $\Hom$ and $\tensor$. Let $U$, $V$ be an $A$-module for an algebra $A$, $\Hom_A(U,V)$ means the set of $A$-module maps from $U$ to $V$. When there is no symbol beneath, then $\Hom(U,V)$ is the set of $\F$-linear maps from $U$ to $V$. We will treat $\Hom(U,V)$ and $U \tensor V$ as $G$-modules automatically. See Section \ref{Sect:Reps} for details.

\section{Categories and Equivalences}

Let $A$ be an algebra. The category of all finitely generated $A$-modules is denoted $A\Mod$. Two algebras $A$ and $B$ are Morita equivalent if $A\Mod$ is categorically equivalent to $B\Mod$. The stable module category of $A$ is denoted by $A\PMod$. We assume readers have basic knowledge about these categories, in particular stable equivalence of Morita type. From these we take for granted the readers will be able to understand this theorem of Linckelmann's, which is crucial to establish our proof.

\begin{theorem}[Linckelmann's Theorem]\label{Trick}
	Let $A$ and $B$ be two self-injective algebras with no simple projective summands. If $A$ and $B$ are stably equivalent of Morita type by an equivalence which sends simple $A$-modules to simple $B$-modules, then $A$ and $B$ are Morita equivalent.
\end{theorem}

Okuyama has used this theorem to prove Brou\'e's abelian defect conjecture for the blocks of $SL_2(p^n)$ \cite{Okuyama}\cite{Yoshii}.

Considering the representation theory of an algebra $A$, another common ground is to work with its derived category of the module category. In our case we would only be interested in the bounded derived category. We write $D^b(A)$ for the bounded derived category of $A\Mod$. We reiterate Rickard's theorem for our case of group algebras (which are self-injective algebras) to emphasis the relation between $D^b(A)$ and $A\Pmod$.

\begin{theorem}[Rickard's Theorem] Let $A$ be a self-injective algebra. The triangulated category $D^b(A)/D^{pc}(A)$ of derived category of $A\Mod$ quotient by perfect complexes of $A\Mod$ is equivalent naturally to $A\Pmod$ as a triangulated category. That is, there exists an equivalence $A\Pmod \to D^b(A)/D^{pc}(A)$ the following square commutes:
\begin{equation*}
\xymatrix{A\Mod \ar@{^{(}->}[r] \ar[d] & D^b(A) \ar[d] \\ A\Pmod \ar@{<->}^-{\sim}[r] & D^b(A)/D^{pc}(A)}
\end{equation*} \end{theorem}

Using this, every derived equivalence $F:D^b(B) \to D^b(A)$ between two self-injective algebras induce a stable equivalence $\overline{F}:B\Pmod \to A\Pmod$. In particular $\overline{F}$ is of Morita type if $F$ is induced by a two-sided tilting complex.

These categories are interesting by their own rights. In here we focus to a particular type of derived equivalence called perverse equivalence.

\subsection{Perverse equivalence}

Perverse equivalence is developed by Chuang and Rouquier to describe a certain type of derived equivalence that can be constructed by some combinatorial data. It has its origins from the construction of perverse sheaves. However, it does not cover all types of derived equivalences, and composition of perverse equivalences might fail to be perverse. Although it has a broad application to various type of categories, we shall only define the perverse equivalences for derived categories of abelian categories to simplify things and allow us bypass some technicalities (such as $t$-structures and hearts). Then we shall give some examples to explain perverse equivalences for module categories of symmetric algebras. We start by the notion of Serre subcategory of an abelian category:

\begin{definition}\label{Serre}
Let $\C$ be an abelian category and $\D$ be a full subcategory. $\D$ is a \emph{Serre subcategory} if given any exact sequence $0 \to K \to 
L \to M \to 0$ in $\C$, $L \in \D$ if and only if $K, M \in \D$. \end{definition} 
Let $\C$ and $\C'$ be two abelian categories. Consider two filtrations \begin{equation*} 0=\C_{-1}\subset \C_0 \subset \dots \subset \C_r=\C \text{ and } 0=\C'_{-1}\subset \C'_0 \subset \dots \subset \C'_r=\C'  \end{equation*} of Serre subcategories and a function $\pi:\{0,...,r\} \to \Z$.
\begin{definition}\label{def:Perv} An equivalence $F:D^b(\C) \to D^b(\C')$ is \emph{perverse relative to $(\C_\bullet, \C'_\bullet, \pi)$} if the following holds:
\begin{itemize} \item $F$ restricts to equivalences $D^b_{\C_i}(\C) \to D^b_{\C'_i}(\C')$. 
\item $F[-\pi (i)]$ induces equivalences $\C_i/\C_{i-1} \to \C'_i/\C'_{i-1}$.\end{itemize} That is, with the natural embedding from $\C_i/\C_{i-1}$ to  $D^b_{\C_i}(\C)/D^b_{\C_{i-1}(\C)}$ we have
\begin{equation*}
\xymatrix{D^b_{\C_i}(\C)/D^b_{\C_{i-1}(\C)} \ar[r]^F  & D^b_{\C_i}(\C)/D^b_{\C_{i-1}(\C)}  \\
	\C_i/\C_{i-1} \ar@{^{(}->}[u] \ar[r]^{F[-\pi (i)]} & \C'_i/\C'_{i-1} \ar@{^{(}->}[u]}
\end{equation*} (c.f. \cite[Definition 2.53]{ChuangRouquier}). \end{definition}
Given an equivalence $F$ perverse relative to $(\C_\bullet, \C'_\bullet, \pi)$, the filtration $\C'_\bullet$ is determined by $\C_\bullet$ and $F$ via $\C'_i=\C'\cap F(D^b_{\C_i}(\C))$.
\begin{proposition}\label{Prop:Perv} Let $F: D^b(\C) \to D^b(\C')$ be perverse relative to $(\C_\bullet, \C'_\bullet, \pi)$.  \begin{enumerate}
\item (reversibility) $F^{-1}$ is perverse relative to $(\C'_\bullet, \C_\bullet, -\pi)$.
\item (composability) Let $F':D^b(\C') \to D^b(\C'')$ be perverse relative to $(\C'_\bullet, \C''_\bullet, \pi')$, then $F' \circ F$ is perverse relative to $(\C_\bullet, \C''_\bullet, \pi+\pi')$.
\item (refineability) Let $\tilde{\C}_\bullet=(0=\tilde{\C}_{-1} \subset \dots \subset \tilde{\C}_{\tilde{r}}$) be a refinement of $\C_\bullet$. Define the weakly increasing map $f:\{0,\dots,\tilde{r}\} \to \{0,\dots,r\}$ such that $\tilde{\C}_\bullet$ collapses to $\C_\bullet$ under $f$ (i.e. $\C_{f(i)-1} \subset \tilde{\C}_i \subset \C_{f(i)}$). Then $F$ is perverse relative to $(\tilde{\C}_\bullet, \pi\circ f)$. 
\item If $\pi=0$ then $F$ restricts to an equivalence $\C \to \C'$. 
\item The information $(\C_\bullet, \pi)$ determine $\C'$ up to equivalence.
\end{enumerate} \end{proposition}
\begin{notation} From 5, since $(\C_\bullet, \pi)$ determine $\C'$ we might sometimes simplify and say a perverse equivalence $F$ is perverse relative to $(\C_\bullet, \pi)$ \end{notation}

\begin{proof} See \cite{ChuangRouquier}. \end{proof}

When every object in an abelian category $\C$ has finite composition series, each object can be broken down to a collection of simple objects components via short exact sequences. Then, by definition, a Serre subcategory is generated by the collection of all simple objects inside it. Thus we can use filtration of simple objects to replace the filtration of Serre subcategories, making the description more concrete. (c.f. \cite[2.2.6]{ChuangRouquier})
\begin{definition} Let $\C$ and $\D$ be abelian categories with finite composition series. Let $\S$ be the set of non-isomorphic simple objects in $\C$. We say that an equivalence $F:D^b(\C) \xrightarrow{\sim} D^b(\D)$ is perverse relative to $(\S_\bullet, \pi)$ when it is perverse relative to $(\C_\bullet, \pi)$ where $\S_\bullet$ is a filtration of isomorphism class of simple objects defined by $\C_\bullet$.  \end{definition} 
\begin{lemma}\label{general}
Let $\C$, $\D$ be abelian categories with finite composition series, \begin{equation*} \S_\bullet=(\emptyset=\S_{-1}\subset \S_0\subset \dots \subset \S_r=\S) \text{ and }\T_\bullet=(\emptyset=\T_{-1}\subset \T_0 \subset \dots \subset \T_r=\T) \end{equation*} be filtrations of isomorphism class of simple objects on $\C$ and $\D$ respectively. Let $p:\{0,\dots,r\} \to \Z$ be a function. An equivalence $F:D^b(\C) \xrightarrow{\sim} D^b(\D)$ is perverse relative to $(\S_\bullet, \T_\bullet, p)$ if for every $i$ the following holds. 
\begin{itemize}
\item Given $M \in \S_i - \S_{i-1}$, the composition factors of $H^r(F(M))$ are in $\T_{i-1}$ for $r\neq -p(i)$ and there is a filtration $L_1 \subset L_2 \subset H^{-p(i)}(I(M))$ such that the composition factors of $L_1$ and of $H^{-p(i)}(F(M))/L_2$ are in $\T_{i-1}$ and that of $L_2/L_1$ are in $\T_i-\T_{i-1}$.
\item The map $M \to L_2/L_1$ induces a bijection $\S_i-\S_{i-1} \xrightarrow{\sim} \T_i-\T_{i-1}$, hence there is a bijection $\beta_F:\S \to \T$.
\end{itemize}\end{lemma}
For proof see \cite[2.64]{ChuangRouquier}.

We put forward an important example of perverse equivalence of symmetric algebras. First we have to define:  
\begin{definition} Let $\S' \subset \S$. Given $M \in A\Mod$, $\phi_M:P_M \to M$ a projective cover. Denote by $M_{\S'}$ the largest quotient of $P_M$ by a submodule of $\ker\phi_M$ such that all composition factors of the kernel of the induced map $M_{\S'} \to M$ are in $\S'$. Similarly for $M \to I_M$ be the injective hull. Denote by $M^{\S'}$ the largest submodule of $I_M$ containing $M$ such that all composition factors of $M^{\S'}/M$ are in $\S'$. \end{definition}

In this (very important) example, we first define a one-sided tilting complex, then we concern how the simple modules are being corresponded,  Since the Serre subcategories are defined by subsets of simple objects as generators.

\begin{example}\label{elementary} Let $A$ be a symmetric algebra. Let $S$ be a simple $A$-module, $P_S$ be the projective cover of $S$. Take $\S'$ to be a subset of isomorphism classes of simple $A$-modules, We define $S_S$, a chain complex of projective $A$-modules depending on $\S'$ as follows. \begin{enumerate} \item If $S \in \S'$, we define \begin{equation*} X_S=(Q_S \xrightarrow{\alpha} P_S \to 0) \end{equation*} where $\alpha$ is a presentation of $S_{\S'}$, $Q_S$ is in degree 0. Note that this forces all composition factors of $\mathrm{head}(Q_S)$ not belong to $\S'$. \item For $S \notin \S'$, \begin{equation*} X_S=(P_S \to 0), \end{equation*} where $P_S$ is in degree 0. \end{enumerate}
Now consider \begin{equation*} X_{\I} := \bigoplus_{S \in \S'} X_S. \end{equation*} It is easy to check (c.f. \cite{Rickard2}) this is a one-sided tilting complex. Setting $B=End_{D^b(A)}(X_{\I})$ we have a functor \begin{equation*} F:D^b(A) \xrightarrow{\sim} D^b(B) \end{equation*} inducing such equivalence. Denote by $\T$ the set of simple $B$-modules. We have a bijection between $\S$ and $\T$ and have $\S'$ correspond to $\T'$, a subset of $\T$. with $F(X_S)=P_T$, the projective cover of $T$ as $B$-modules.

Now considering $\Hom_A(X_{\I}, S)$ for all $S \in \S$ we have \begin{equation*} F(S)=\begin{cases} T[-1] &\text{ if } S \in \S'  \\ T^{\T'} &\text{ otherwise} \end{cases} \qquad \text{and} \qquad F^{-1}(T)=\begin{cases} S[1] &\text{ if } T \in \T'  \\ S_{\S'} &\text{ otherwise.} \end{cases}\end{equation*} Since the set of non-isomorphic simple modules generates the derived category as triangulated category and subsets of simple modules generate Serre subcategories, we can also regard that the perverse equivalence is defined by such a map. \end{example}
\begin{remark} This example characterises elementary perverse equivalences for symmetric algebras. $F$ is perverse relative to ($0 \subset \S' \subset \S$, $0 \subset \T' \subset \T, \varepsilon:\{0 \to 1; 1 \to 0\}$). See \cite[2.71]{ChuangRouquier} (shifted by 1 globally on perversity function).\end{remark}

In order to handle compositions of perverse equivalences described in the paper, we introduce the notion of poset perverse equivalence. Mentioned by Chuang-Rouquier in \cite{ChuangRouquier} and explored by Dreyfus-Schmidt\cite{Leo}, poset perverse equivalence allows a better perspective to consider compositions of perverse equivalences. 

\begin{definition}\label{def:poset}
	Let $\C$ be an abelian category with $\S$ the non-isomorphic set of simple objects. Let $\D$ be another abelian category with $\T$ the non-isomorphic set of simple objects. A derived equivalence $F:D^b(\C)\to D^b(\D)$ is perverse relative to $(\S, \prec, \omega)$, where $\prec$ is a partial order on $\S$ and $\omega:\S \to \Z$, if and only if
	\begin{enumerate}
		\item There is a one-to-one correspondence $\beta_F :\S \to \T$.
		\item Define $S_{\prec}=\{T \in \S\mid T\prec S\}$. The composition factors of $H^r(F(S))$ are in $\beta(S_{\prec})$ for $r\neq -\omega(S)$ and there is a filtration $L_1 \subset L_2 \subset H^{-\omega(S)}(F(S))$ such that the composition factors of $L_1$ and of $H^{-\omega(S)}(F(S))/L_2$ are in $\beta(S_{\prec})$ and $L_2/L_1$ is isomorphic to $\beta (S)$.
	\end{enumerate}
	When considering a perverse equivalence $F: D^b(\C)\to D^b(\D)$ we can transfer the partial order on $\S$ to $\T$ via $\beta_F$.
\end{definition}

One obvious problem of introducing a new definition is the compatibility of two notions. It is not hard to notice that they are interchangeable, provided that the poset perverse equivalence exists.

\begin{lemma}
	A derived equivalence $E:D^b(\C) \to D^b(\D)$ that is perverse relative to $(\S, \prec, \omega)$ is also perverse relative to $(I_{\bullet}, \phi)$ for a certain filtration on simple objects $I_{\bullet}$ and perversity function $\phi$. On the other hand, given a derived equivalence $E$ perverse relative to $(I_{\bullet}, \phi)$, $E$ is also perverse relative to $(\S, \prec, \omega)$ for certain partial order $<$ and perversity function $\omega$.
\end{lemma}
\begin{proof}
	If $E$ is a poset perverse equivalence with respect to a partial order $\prec $, refine to a total order and let $I_{\bullet}$ be the corresponding maximal filtration on $\S$. That is, for each $i$, $I_i-I_{i-1}=\{S_i\}$ for a certain $S_i$. Now define perversity function $\phi(i)=\omega(S_i)$. Conversely, if $E$ is perverse relative to $(I_{\bullet}, \phi)$, we define a partial order on $\S$: Define $S_i\prec S_j$ if and only if there exist a layer $I_k$ such that $S_i \in I_k$ and $S_j \notin I_k$. Now define $\omega(S_i)=\phi(k)$ for the only $k$ satisfying $S_i \in I_k-I_{k-1}$. One can easily check each definition makes each fulfil the other description of perverse equivalence.
\end{proof}

\begin{definition}
	We say that the datum $(\S, \prec, \omega)$ is compatible with the datum $(I_\bullet, \phi)$ if the following conditions hold \begin{enumerate} \item If $S_a \prec S_b$ then there exists an $i$ such that $S_a \notin I_i$, $S_b \in I_i$. \item If $S_a, S_b \in I_i-I_{i-1}$ then $\omega(S_a)=\omega(S_b)=p(i)$. \end{enumerate}
\end{definition}

Let $E:D^b(\C) \to D^b(\D)$ be a perverse equivalence (filtered or poset). Consider image of $E(S)$ with $S$ running through all simple objects in $\C$. We can define a partial order $\prec$ by \begin{equation*}
S \prec S' \text{ if } H^*(E(S')) \text{ contains a copy of }S \text{ in its composition factors.}
\end{equation*}This is defined as the coarsest partial order on $E$. 
We call a proposition similar to \ref{Prop:Perv} for the properties of poset perverse equivalence.

\begin{proposition}\label{Prop:PosPerv}
	Let $\C$ be an abelian category with finite composition series and a complete set of non-isomorphic simple objects $\S_{\C}$. Let $E:D^b(\C) \to D^b(\C')$ be an equivalence perverse relative to $(\S_{\C}, \prec, \omega)$. \begin{enumerate}
		\item $E^{-1}$ is perverse relative to $(\beta_E(\S_{\C}), \beta_E(\prec), -\omega\circ \beta_E^{-1})$.
		\item Let $E':D^b(\C') \to D^b(\C'')$ be perverse relative to $(\beta_E(\S_{\C}), \beta_E(\prec), \omega')$ then $E'\circ E$ is perverse relative to $(\S_{\C}, \prec, \omega+\omega')$
		\item If $\omega=0$ then $E$ restricts to an equivalence $\C \to \C'$.
	\end{enumerate}
\end{proposition}

\begin{proof}
	The proof is the same as Proposition \ref{Prop:Perv} except for item 2. Consider the homology $H^*(E(S))$ of the image of a simple object $S$ under $E$. The composition factors of $H^*(E(S))$ contain a copy of $\beta_E(S)$ in the $-\phi_I(S)^{th}$ degree and the rest of factors $R$ satisfy $\beta^{-1}_ER \prec S$. Similarly, the homology $H^*(E'E(S))$ has a copy of $\beta_{E'}\beta_E(S)$ at degree $-\omega(S)-\omega'(S)$, and all the remaining composition factors $R'$ of the homology satisfy $\beta_E^{-1}\beta_{E'}^{-1}R' \prec S$.
	Then by Definition \ref{def:poset} the derived equivalence $E:D^b(\C) \to D^b(\C'')$ we have constructed is perverse relative to $(\S, \prec, \omega+\omega')$.
\end{proof}

\section{Representations of $G$ and $H$}\label{Sect:Reps}

The general block theory of a group algebra $\FG$ and its subgroups are described by Brauer's theorems\cite{Alperin}. Here we concentrate on the full defect block(s) of $\FG$. It has a Sylow $p$-subgroup $P$ as its defect group, hence these blocks (and $\FG$-modules lying in them) correspond to blocks of $H=N_G(P)$ (and $\FH$-modules lying in them). In our case, we have 'trivial intersection' condition, that is, $P \cap gPg^{-1}$ is either $P$ or 1. When this holds, the Green correspondence reads \begin{theorem}\label{Green}(Green correspondence for trivial intersections) Let $G$ be a group and $P$ be a Sylow $p$-subgroup and let $L=N_G(P)$. Then there is a one-to-one correspondence between the isomorphism classes of non-projective $\FG$-modules $U$ and the isomorphism classes of non-projective $\F L$-modules $V$ such that \begin{equation*} U\hspace{-.4em}\downarrow_L \isom V \oplus Q\end{equation*} \begin{equation*} V\hspace{-.4em}\uparrow^G \isom U \oplus P \end{equation*} where $P$, $Q$ are projective $\FG$ and $\F L$-modules respectively. \end{theorem}

This leads to
\begin{equation*} \PHom_{\FG}(M,N)\isom\PHom_{\FH}(M\hspace{-.4em}\downarrow_H,N\hspace{-.4em}\downarrow_H). \end{equation*}
Hence, $\FG\Pmod$, $\B\Pmod$ and $\FH\Pmod$ are stably equivalent. This equivalence, since given by induction and restriction functors, are of Morita type. 

\subsection{$\FG$-modules and $\FH$-modules}
We start by describing the simple $\FG$-modules. Let $V$ be the natural two dimensional representation of $\FG$. Denote by $V^i$ the $i^{th}$ symmetric power of $V$. (Note it is NOT the tensor product of $i$ copies of $V$.)  Define the $j^{th}$ Frobenius twist on $V^i$, $\sigma^j(V^i)$, via the Frobenius automorphism on $G$. $\sigma^j(V^i)$ is the $G$-module whose underlying space is $V^i$ but for $g \in G$, the action is defined as $g(v')=\sigma^j(g)(v)$ for all $v$. 

Steinberg tensor product theorem describes all the simple modules of $SL_2(q)$.
\begin{definition} For $0 \leq a \leq q-1$ define 
\begin{equation}\label{STPT} S_a=\bigotimes_{i=0}^{n-1}\sigma^i(V^{a_i}). \end{equation}
\end{definition}
\begin{theorem}
	$S_a$ are simple for $0 \leq a \leq q-1$, mutually non-isomorphic. They form a complete set of mutually non-isomorphic simple $\FG$-modules.
\end{theorem}
Furthermore, $S_{q-1}$ is (simple and) projective. In block theory of $SL_2(q)$, $p^n$ copies of $S_{q-1}$ forms a block of defect zero which we called the Steinberg block in $SL_2(q)$. The block itself is not very much of our concern since its Brauer correspondent is itself. However, we will use the fact that $S_{q-1}$ is projective in some (strange) manner later.

For $p=2$, all the remaining simple modules fall into one full defect block, the principal block $\B_0$. For odd primes they fall into two distinct full defect blocks. The principal block, $\B_0$, has all evenly numbered simple modules $S_0$, $S_2$,..., $S_{q-2}$ and the non-principal block $\B_1$ consists of all oddly numbered simple modules $S_1$, $S_3$,..., $S_{q-2}$. To unify the description disregarding parity of primes, 
\begin{definition} The direct sum the full defect blocks of $\FG$ is denoted $\B$. Denote the complete set of non-isomorphic simple $\B$-modules by \begin{equation*}
	\S=\{S_a\mid 0 \leq a \leq q-2\}
	\end{equation*} \end{definition}
\begin{remark} $\B$ is the algebra such that $\FG= St \oplus \B$, where $St$ is the Steinberg block. \end{remark}

\begin{definition}
For an integer $a$, $0\leq a \leq q-1$, denote $M_a$ the $\FH$-module $S_a\hspace{-.4em}\downarrow_H$ given by restricting the corresponding simple $\FG$-module. 
\end{definition}

Now we discuss the representation theory of group $H$, all of whose block(s) is(are) the local Brauer correspondent(s) of the full defect block(s) of $\FG$. $H$ as a group is $C_q^n \rtimes C_{q-1}$. It is quite easy to obtain its simple modules - they are all 1-dimensional. Let $\alpha$ be a generator of $\F_q$. Define $U_i$ to be the 1-dimensional $\FH$-module on which $\begin{pmatrix}
\alpha^{-1} & *\\
0 & \alpha
\end{pmatrix}$ acts on $U_i$ by multiplying every vector by $\alpha^i$.
It is obvious that $U_i$ is isomorphic to $U_j$ if and only if $i\equiv j \pmod{q-1}$. Every simple $\FH$-module arises in this way.

\begin{remark} Most literature (except Holloway in the list of referenced authors) define the simple $\FH$-modules $U_i$ by having the matrix $\begin{pmatrix}
\alpha^{-1} & *\\
0 & \alpha
\end{pmatrix}$ acts on $U_i$ by multiplying $\alpha^{-i}$ instead. In other words the conventional $U_i$ defined in other literature will be our $U_{-i}$ instead. We use Holloway's convention to avoid negative signs arise later. \end{remark}

The Frobenius automorphism of $G$ restricts to an automorphism to $H$. So it twists also $\FH$-modules. Simple calculation shows that $\sigma(U_i)\isom U_{pi}$.

The block structure of $\FH$ is similar to $\B$. Concretely, for $p=2$, the whole group algebra $\FH$ forms a single block. For odd primes, $\FH$ decomposes into two blocks, the principal block containing evenly numbered simple modules $U_0$, $U_2$,..., $U_{q-3}$ and non-principal block containing oddly numbered simple modules $U_1$, $U_3$,..., $U_{q-2}$. 

Consider $U_i \tensor -$ as a functor from $\FH\Mod$ to itself. Since $U_i$ is one-dimensional, the following conclusion can be easily checked: \begin{enumerate}
\item $U_i \tensor U_j \isom U_{i+j}$;
\item $U_i \tensor - $ induces a Morita self-equivalence with inverse functor $U_{-i}\tensor - $;
\item The endofunctor on $\FH\Pmod$ induced by $U_i \tensor -$ is exact. 
\end{enumerate}

\begin{notation} We omit the tensor product symbol from $U_i \tensor_{\F} -$ for convenience. For the rest of this paper we almost always treat $U_i$ as a functor.  
\end{notation} 

The ultimate aim is to explore extensions in $\FG\Mod$. A natural choice is to look at distinguished triangles in $D^b(\FG)$. However, it turns out to be extremely difficult. In fact, the piece of information in question is too cryptic in $\FG\Pmod$. Luckily in $\FH\Pmod$, we have the nice series of functors $U_i \tensor - $ to aid calculations which is enough for our job. In order to do so we consider the restriction of simple $\FG$-modules to $\FH$-modules. Using Steinberg tensor product theorem (c.f. \ref{STPT}) and the fact that the restriction of $V^i$ is a uniserial $\FH$-module with 1-dimension components for $0 \leq i \leq n-1$ (see Chapter 5 of \cite{Holloway} or \cite{Chuang} for its proof), it is not hard to obtain the structure of these restrictions. They are indecomposable modules with a 'hypercuboid shape', see either \cite{Holloway} or the Appendix for details.

The remaining sections in this chapter consists of the needed work of tailoring the intrinsic structure of the related categories into useful lemmas and corollaries. These calculation of possible extensions for all possible cases are highly combinatorial, but these cases can be simplified into two lemmas, with the extensions represented by certain distinguished triangles in $\FH\Pmod$. 

\subsection{Triangles in the stable module categories of blocks}

In this section, we fix $n$ to be greater than 1. All tensor products are over $\F$ unless otherwise stated.
\begin{definition}\label{Aux} Consider the $i^{th}$ digit of the base-$p$ representation $a_i$ with $0 \leq a_i \leq p-2$.\begin{enumerate} \item Define $a'_i$ to be $p-2-a_i$; \item For an integer $a$ with  $0 \leq a_i \leq p-2$ for some $i$, \begin{itemize} \item define $a(i')=(a_0,...,a_{i-1},a'_i,a_{i+1},...,a_{n-1})$ to be the number acquired by replacing the digit $a_i$ by $a'_i$. \item $a(\overline{i})=(a_0,...,a_{i-1},p-1,a_{i+1},...,a_{n-1})$ be the number acquired by replacing the digit $a_i$ by $p-1$. \end{itemize} \end{enumerate} \end{definition}
\begin{remark} Note that $a''_i=a_i$. Using this we will define a pairing (later) between integers $0 \leq a \leq q-2$ which we will often use later. \end{remark}

We are going to build up some lemmas, culminating to a general description of certain distinguished triangles of $\FH\Pmod$ for further calculation.
\begin{lemma}\label{SES}
Let $i$ be an integer with $0 \leq i \leq n-1$. Let $l$ be any integer and $t$ be an integer with $0\leq t \leq p-2$. Let $j=l-p^i(p+1+t)$ and $\breve{j}=l-p^i(p-1-t)$. Then any non-zero $\FH$-homomorphism from $U_jV^{p-1}$ to  $U_{\breve{j}}V^{p-1}$ has a cokernel isomorphic to $U_lV^t_i$.
\end{lemma}
\begin{proof} See Lemma 4 of \cite{Chuang} (further referenced to Lemma 2.1 and 2.2 of \cite{Carlson}). Note the proof from \cite{Chuang} can be directly adapted for arbitrary $i$. Also note $U_i$ in \cite{Chuang} becomes $U_{-i}$ here. \end{proof}

\begin{lemma} We have non-split short exact sequences of the following $\FH$-modules:
\begin{equation*} 0 \to U_{-p^{i+1}}V^{b'_i}_i \to U_{-p^i(p-1-b_i)}V^{p-1}_i \to V^{b_i}_i \to 0 \end{equation*}
\begin{equation}\label{iFill} 0 \to U_{-p^{i+1}}M_{b(i')} \to U_{-p^i(p-1-b_i)}M_{b(\overline{i})} \to M_b \to 0. \end{equation} \end{lemma}
Recall that $M_a$ is the restriction of $S_a$ from $G$-modules to $H$-modules.
\begin{proof}
The two short exact sequences can be obtained similarly to the proof of Lemma 6 of \cite{Chuang}: Using the previous lemma we have an exact sequence
\begin{equation*}U_jV^{p-1}_i \to U_{\breve{j}}V^{p-1}_i \to U_{p^i(p-1-b_i)}V^{p-1}_i \to V^{b_i}_i \to 0 \end{equation*}
where $j=-p^i(p-1-b_i+2p)$ and $\breve{j}=-p^i(p+1+b_i)$. Using the previous lemma, the first homomorphism has a cokernel isomorphic to $U_{-p^{i+1}}V^{b'_i}$ which gives the first sequence. \\
The second sequence is obtained by tensoring the first sequence at each term with $V^{b_0}_0,...,V^{b_{n-1}}_{n-1}$ except $V^{b_i}_i$. Then using (\ref{STPT}) and restriction to see this is the desired result. \end{proof}

\begin{lemma}\label{TRI} Let $b_i$ be an integer with $0 \leq b_i \leq p-2$ with $0 \leq i \leq n-1$.  Then we have the following triangle in $\FH\Pmod$:
\begin{equation}\label{ExtMb}U_{p^i(1+b'_i)} \Omega M_{b(\overline{i})} \to U_{p^{i+1}}\Omega M_{b(i')} \to M_b \leadsto \end{equation}
\end{lemma}

\begin{proof}
The short exact sequence (\ref{iFill}) in $\FH\Mod$ induces a triangle 
\begin{equation*}  U_{-p^{i+1}}M_{b(i')} \to U_{-p^i(p-1-b_i)}M_{b(\overline{i})} \to M_b \leadsto \end{equation*}
in $\FH\Pmod$.
To obtain the stated triangle from the short exact sequence, we take it as a triangle in the stable module category and perform the following steps:\begin{enumerate} 
\item Tensor throughout the triangle obtained by $U_{p^{i+1}}$,
\item Relabel $b_i$ by $b'_i$ (and vice versa).
\item Rotate the triangle (c.f. definition 1.63(c)) two places to the left. (Put the rightmost term to the leftmost and shift by $\Sigma^{-1}$. In stable $\FH$-module category $\Sigma^{-1}$ is represented by applying $\Omega$; perform this twice.).
\end{enumerate} \end{proof}
\begin{remark} This triangle does lie in a particular block of $\FH$, depending on the parity of $b$. When $M_{b(\overline{i})}$ is the Steinberg module (i.e. $b(\overline{i})=q-1$), we regarded that module as zero module. That is partly justified by the fact the Steinberg module restricts to a projective module. \end{remark}

In particular we have:
\begin{lemma}\label{PerProj} Fix $i$ to be an integer with $0 \leq i \leq n-1$. Let $b_i$ be integers such that $0 \leq b_i \leq p-2$. $m$ is an integer. Then 
$\begin{cases}
\Omega^m M_{(p-1,...,b_i,...,p-1)}\isom U_{-mp^{i+1}}M_{(p-1,...,b'_i,...,p-1)}&\text{ if } m \text{ is odd.} \\ 
\Omega^m M_{(p-1,...,b_i,...,p-1)}\isom U_{-mp^{i+1}}M_{(p-1,...,b_i,...,p-1)} &\text{ if } m \text{ is even.} 
\end{cases}$
\end{lemma}
\begin{proof}
We only need to prove $\Omega M_{(p-1,...,b_i,...,p-1)}=U_{-p^{i+1}}M_{(p-1,...,b'_i,...,p-1)}$, since both cases are only the $m^{th}$ iteration of it. Using (\ref{iFill}) with all digits as $p-1$ except $b_i$ we have 
\begin{equation*} 0 \to U_{-p^{i+1}}M_{(p-1,...,b'_i,...,p-1)} \to U_{-p^i(p-1-b_i)}M_{p^n-1} \to M_{(p-1,...,b_i,...,p-1)} \to 0. \end{equation*}
Since the middle term is projective it is regarded as zero in stable module category. Hence the first term is isomorphic to the Heller translate of the last term by the axioms, which is exactly the desired equation.
\end{proof}
\begin{remark} The highlight in the preceding lemma is the $-mp^{i+1}$ subscript of $U$ regardless of whether subscript of $M$ has $b_i$ or $b'_i$ as its $i^{th}$ digit. \end{remark}

\subsection{Extension lemmas}

We have to determine the possible extension of some $\FG$-modules. This is for us to work through perverse equivalence later. To achieve this, we transfer $\FG$-modules to $\FH$-modules using the restriction functor. Since \begin{equation*}\label{ExtHom}\Ext^1_{\FG}(M, N)=\PHom_{\FG}(\Omega M, N)=\PHom_{\FH}(\Omega M\hspace{-.4em}\downarrow_H, N\hspace{-.4em}\downarrow_H), \end{equation*} to find out the necessary extension needed in our construction later we introduce two lemmas. First, we follow the route in \cite{Chuang} and utilise Carlson's calculations on $\Ext$ groups of simple $\FG$-modules \cite{Carlson}, then we adapt the result to $\FH\Pmod$ using the restriction functor and generalises it. We end up with a refinement similar to the one in \cite[lemma 6]{Chuang}. Second, we need another piece of information which turns out to be a direct calculation of stable homomorphism group ($\PHom$) of some modules, which generalises a lemma from Holloway \cite{Holloway}.

\begin{lemma}\label{Carlson}
For $0 \leq i \leq n-1$, suppose $b_i, c_i$ ranges from 0 to $p-1$ and $j,\breve{j}$ are integers. Then the dimension of \begin{equation*} \Ext^1_{\FH}(U_jM_b, U_{\breve{j}}M_c) \isom \PHom_{\FH}(\Omega U_jM_b, U_{\breve{j}}M_c) \end{equation*}  is determined by the number of $n$-tuples $(l, k_0,...,k_{n-1})$ of integers satisfying: 
\begin{equation*} b_l,c_l \leq p-2, \end{equation*}
\begin{equation}\label{ExtEq1}
j-\breve{j}+\sum_{\substack{i=0\\i \neq l}}^{n-1}p^i(b_i-c_i+2k_i)+p^l(-b_l-c_l+2k_l-2) \equiv 0 \pmod {p^n-1}
\end{equation}
with also
\begin{equation*} \mathrm{max}\{0, c_i-b_i\} \leq k_i\leq c_i \end{equation*} for $i \neq l$ and 
\begin{equation*} \mathrm{max}\{0, b_l+c_l+2-p \} \leq k_l\leq \mathrm{min}\{b_l, c_l\}. \end{equation*}
\end{lemma}
\begin{proof}
The proof is similar to the proof in \cite{Chuang}. Consider the summation \begin{equation}\sum_{\substack{i=0\\i \neq l}}^{n-1}p^i(b_i-c_i+2k_i)+p^l(-b_l-c_l+2k_l-2) \equiv 0 \pmod {p^n-1}.
\end{equation}Adapt Carlson's theorem \cite[Theorem 4.1]{Carlson} in $\Ext$ groups of simple $\FG$-modules. Fixing $r=1$, it splits into two cases. \begin{enumerate} \item When $p$ is odd, condition (1) forces $e_i=0$ and $f_i=0$ except for one $f_i$, record this subscript as $l$. Condition (3) gives the first constraint, and the condition \ref{ExtEq1} is a simplified version after substitution. \item When $p=2$, fixing $r=1$ forces all but one $e_i=0$. Again we record that subscript $l$, this forces $b_l=c_l=0$. The requirements on $k_i$ in our version is a precise replacement of (3') in \cite{Carlson}. Then to see the last two equations here agree with the original, note we factored the first term $-2(2^l)$ (since $e_l=1$) into the summand to yield the $-2$ in $p^l$ term. With the fact that $b_l=0$, the terms inside bracket of $p^l$ are indeed equal. \end{enumerate} Finally, the $j-\breve{j}$ term in (\ref{ExtEq1}) is introduced using a spectral sequence argument as in \cite{Chuang}, and our proof is complete. \end{proof}

\begin{lemma}\label{Holloway}
The dimension of the stable morphism group $\PHom_{\FH}(U_jM_b, U_{\breve{j}}M_c)$ is equal to the number of $n$-tuples $(k_0,...,k_{n-1})$ of integers satisfying:
\begin{enumerate}
\item $\mathrm{max}\{0,c_i-b_i\}\leq k_i \leq c_i$ for all $i$.
\item There exist an integer $l$ such that $k_l < b'_l$.
\item \begin{equation}\label{ExtEq2} j-\breve{j}+\sum_{i=0}^{n-1}p^i(b_i-c_i+2k_i) \equiv 0 \pmod {p^n-1} \end{equation}
\end{enumerate}
\end{lemma}
\begin{proof}
Considering the restriction of $\FG$-simple modules are of a special class of $\FH$-modules with the shape of hypercuboids. The $\FH$-modules $M_b$ has irreducible top $U_b$ and the length of its sides $(1+b_0, 1+b_1,..., 1+b_{n-1})$. All components within the cuboid is decided by its position (c.f. \cite[pg.35]{Holloway}). Now consider $\PHom_{\B}(U_jM_b, U_{\breve{j}}M_c)$, it has become a consideration of the head of $U_jM_b$'s position in $U_{\breve{j}}M_c$. A two-dimensional illustration (cuboid becomes rectangle) is shown in \cite[figure 5.2]{Holloway}. 

Condition (1) restricts the position of the modules such that $U_jM_b$ contains the socle of $U_{\breve{j}}M_c$. Condition (2) rules out the possibility of such a map factoring through the injective hull of $U_jM_b$. Shown by dashed line in the figure, if the injective hull, which is known to have size $(p-1,...,p-1)$, covered $U_{\breve{j}}M_c$, the map factors through projectives (=injectives) hence quotiented out of $\PHom(U_jM_b, U_{\breve{j}}M_c)$. Lastly condition (3) locate the head of $U_jM_b$ in the component of $U_{\breve{j}}M_c$.
\end{proof}
\begin{remark} The proof is a generalisation of \cite[Theorem 5.2.1 (2)]{Holloway}. \end{remark}

These two lemmas build up arithmetic constraints for a certain type of extension. Now we would tailor the lemmas into two of particular situation. But first we have to decode the modulo equations (\ref{ExtEq1}) and (\ref{ExtEq2}) in both of the lemmas. Temporarily ignore the term $j-\breve{j}$ in (\ref{ExtEq1}) and (\ref{ExtEq2}) and regarded $p$ as an indeterminate. We define the \emph{$p^i$-digit} to be the coefficient with the term $p^i$ (as if $p$ is an indeterminate). The following two inequalities aim at looking at these $p^i$-digits.
\begin{lemma} With $b_i$, $c_i$, $k_i$, $l$, $p$ as defined and restricted under lemma \ref{Carlson} and lemma \ref{Holloway}, we have
\begin{equation}\label{POS} 0 \leq |b_i-c_i| \leq b_i-c_i+2k_i \leq b_i+c_i \leq 2p-2 \end{equation}
\begin{equation}\label{NEG} -p \leq -b_l-c_l+2k_l-2 \leq -2. \end{equation}
\end{lemma}
\begin{proof}
This statement is a technicality mentioned but not shown explicitly in \cite[Theorem 4.1]{Carlson}. Without loss of generality assume $b_i\geq c_i$. For the inequality signs in (\ref{POS}), the first sign is obvious, second sign because $0 \leq k_i$, third sign because $k_i \leq c_i$ and fourth as $b_i, c_i \leq p-1$. We turn to (\ref{NEG}) and the maximum value is \begin{equation*} -b_l-c_l+2c_l-2=c_l-b_l-2\leq -2. \end{equation*} Considering the minimum value of (\ref{NEG}), we split into two cases: \begin{enumerate} \item When $b_l+c_l\leq p-2$ we have $-b_l-c_l-2\geq -p$
\item When $b_l+c_l > p-2$ we have $-b_l-c_l+2(b_l+c_l+2-p)-2=b_l+c_l+2-2p\geq -p$. \end{enumerate} Combining all the arguments gives the two inequalities.
\end{proof}

In the next part we will be defining some new symbols and terms that are needed to express clearly the upcoming results. These end up with two corollaries of Lemma \ref{Carlson} and Lemma \ref{Holloway}, which show that the triangles (\ref{ExtMb}) are exactly what are needed to verify our main theorem. The proof of the arguments are much like \cite[Theorem 4.1]{Carlson} with extra consideration for subscripts of $U$ expressed in the statement by the term $j-\breve{j}$.

\begin{definition}\label{Filt}
Recall $\S$ is the set of non-isomorphic simple $\B$-modules (Definition 2.15).
Define sets \begin{equation*} I_i=\{S_a \mid a_{i+1}=...=a_{n-1}=p-1\}\end{equation*} for $0 \leq i \leq n-1$ to be subsets of simple $\B$-modules. \end{definition}
Note that $I_{n-1}=\S$ since we have no restriction on $S_a$, and a filtration \begin{equation*} I_\bullet = (\emptyset=I_{-1}\subset I_0 \subset I_1 \subset \ldots \subset I_{n-1}=\S) \end{equation*} on the complete set of non-isomorphic simple $\B$-modules.

\begin{definition}\label{Layer}
Fix a prime $p$. We say a simple module $S_a$ is \emph{in layer $i$} if $S_a \in I_{i+1}-I_i$. We also define that an integer $a$ and the $\FH$-module $M_a$ are \emph{in layer $i$} if $S_a$ is.\end{definition}
\begin{remark} This is equivalent to saying $a$ has base-$p$ representation $(a_0,...,a_i,p-1,...,p-1)$ with $a_i\neq p-1$, or by base-$p$ arithmetic, \begin{equation*}
	p^n-p^{i+1} \leq a \leq p^n-p^i-1.
	\end{equation*} 
\end{remark}
\begin{definition}\label{Corr}
Fix a prime $p$. Let $a$ be an integer with $0\leq a \leq q-2$ with base-$p$ representation $(a_0,...,a_{n-1})$. Let $a$ be an integer in the layer $s$. That is, we have\begin{equation*} a=(a_0,...,a_s,p-1,...,p-1)\qquad \text{with} \qquad a_s\leq p-2. \end{equation*} The \emph{partner} of $a$, denoted $a'$ (see notation below for clarification of use), is \begin{equation*} a(s')=(a_0,...,a'_s,p-1,...,p-1);\end{equation*} The \emph{completion} of $a$, denoted $\overline{a}$, is \begin{equation*} a(\overline{s})=(a_0,...,a_{s-1},p-1,p-1,...,p-1).\end{equation*} (c.f. Definition \ref{Layer} and Definition \ref{Aux}) 
\end{definition}
\begin{notation}Recall that we defined $a'_i=p-s-a_i$. It will not contradict if we apply the following: If there is a subscript on the letter concerned, the prime treats it as a base-$p$ digit i.e. $a'_i=p-2-a_i$. Otherwise it is treated as the partner of the integer defined just above, i.e. $a'=(a_0,...,a'_i,...,a_{n-1})$ for $i$ the layer of $a$. \end{notation}
\begin{remark}The partner defined here turns out to be the correspondence of simple $\B$-modules used for our trick later. For odd primes, an even number is a partner of an odd number of the same layer, and vice versa. For $p=2$, the partner of every integer is itself. We also point out that under this involution, the filtration in Definition \ref{Layer} is fixed. \end{remark}

\begin{definition}\label{FC}
For an integer $m$ with $1\leq m \leq p^{n-1}$, define $r_m$ to be the integer such that $p^{r_m}$ divides $m$ with a $p'$-integer quotient. Let \begin{equation*} m=\sum_{i=0}^{n-1}p^im_i=\sum_{i=r_m}^{n-1}p^im_i\end{equation*} be its base-$p$ expression. For an integer $s$, define \begin{equation*} \lfloor m \rfloor_s:=\sum_{i=s}^{n-1}p^im_i,\end{equation*} the \emph{floor of $m$ at $s$} and \begin{equation*} \lceil m \rceil^s=\sum_{i=s}^{n-1}p^im_i+p^s, \end{equation*} the \emph{ceiling of $m$ at $s$}. 
\end{definition}
\begin{remark} We hide the subscript $m$ of $r$ when it is obvious which $m$ we are referring to. Note that for any $m$, the floor(resp. ceiling) of $m$ at $s$ is the nearest integer smaller(resp. greater) than or equal to $m$ that is divisible by $p^s$. \end{remark}

\begin{proposition}\label{ExtCor1}
Let $m$ be an integer with $1\leq m \leq p^{n-1}$. Let $M_c$ be a module in layer $i$ for some $i\leq r=r_m$, and $M_b$ be a module in layer $s$ with $s>r$.
\begin{enumerate}[(a)]
\item If $m$ and $b$ satisfy $m_{s-1}+b_s < p-1$, then \begin{equation*} \PHom_{\FH}(U_{\lfloor m \rfloor_s p}M_b, U_{mp}M_c)=0; \end{equation*}
\item If $m$ and $b$ satisfy $m_{s-1}+b_s = p-1$ and $m_{s-2}=...=m_0=0$, then \begin{equation*} \PHom_{\FH}(U_{\lfloor m \rfloor_s p}M_b, U_{mp}M_c) \end{equation*} is of dimension 1 when $c=\overline{b}$. The corresponding unique non-split extension of $U_{mp}M_c$ by $U_{\lfloor m \rfloor_s p}M_b$ is represented by a distinguished triangle \begin{equation*} U_{mp}\Omega M_c \to U_{\lceil m \rceil^s p}\Omega M_{b'}\to U_{\lfloor m \rfloor_s p}M_b \leadsto \end{equation*}
\end{enumerate} in $\FH\Pmod$.
\end{proposition}
\begin{proof}
The condition (\ref{ExtEq2}) in Lemma \ref{Holloway} requires 
\begin{equation*} \lfloor m \rfloor_s p-mp+\sum_{i=0}^{n-1}p^i(b_i-c_i+2k_i) \equiv 0 \quad \pmod{p^n-1} \end{equation*} for some $k_i$ satisfying condition 1 in Lemma \ref{Holloway} and a particular $k_l$ for condition 2 in Lemma \ref{Holloway} if a non-zero stable homomorphism exists. Note that \begin{equation*} \lfloor m \rfloor_s p-mp=-\sum_{i=r}^{s-1}p^{i+1}m_i=-\sum_{i=r+1}^{s}p^im_{i-1}. \end{equation*} Merging the term $mp-\lfloor m \rfloor_s p$ (the $j-\breve{j}$ term) into the last expression, we have that
\begin{equation}\label{1} \sum_{i=0}^{r}p^i(b_i-c_i+2k_i)+\sum_{i=r+1}^{s}p^i(b_i-(p-1)+2k_i-m_{i-1})+\sum_{i=s+1}^{n-1}p^i(2k_i)\end{equation} has to be divisible by $p^n-1$. Now we consider the actual value of expression (\ref{1}). From (\ref{POS}), the $p^{r+1}$ to $p^s$ digits lie between $-(p-1)$ and $2(p-1)$ and other digits lie between 0 and $2(p-1)$. Hence, for a possible extension to exist, (\ref{1}) evaluates to one of the four following values: 0, $p^n-1$, $2p^n-2$ or $-p^n+1$. Firstly, it cannot be $-(p^n-1)$ because it requires (\ref{1}) to have all terms at $-(p-1)$ (the smallest possible upon all summands) but it must have a non-negative $p^0$ term. Secondly note that $c_l=p-1$ will force $b'_l<p-1-b_l \leq k_l$. So, the range of $l$ mentioned in condition 2 of theorem \ref{Holloway} is restricted to $0 \leq l \leq r$. For this particular $l$, we have \begin{equation*} b_l-c_l+2k_l \leq b_l-c_l+b'_l+c_l=p-2. \end{equation*} This rules out the possibility for (\ref{1}) to be $2p^n-2$. Thirdly, the last inequality indicates the sum up to $p^l$-term:
\begin{equation*} \sum_{i=0}^{l}p^i(b_i-c_i+2k_i) \end{equation*} (since $l$ cannot be greater than $r$ as $b_l, c_l\leq p-2$) lies between 0 and $p^{l+1}-2$, which can never be $-1$ modulo $p^{l+1}$. Adding up the remaining terms of (\ref{1}) will not change this. However, if the expression (\ref{1}) is equal to $p^n-1$, we have the expression equal to $-1$ modulo $p^{l+1}$, creating a contradiction. Thus, the argument above boils down to the conclusion that (\ref{1}) is zero. Now we split into the following two cases:
\begin{enumerate}[(a)]
\item If the condition $m_{s-1}+b_s < p-1$ holds, it forces the value of the $p^s$-digit to be at least 1. Since the sum \begin{equation} \sum_{i=0}^{r}p^i(b_i-c_i+2k_i)+\sum_{i=r+1}^{s-1}p^i(b_i-(p-1)+2k_i-m_{i-1})\end{equation} must be greater than $-(p^s-p)$, adding the next term $p^s$ will make the subtotal \begin{equation} \sum_{i=0}^{r}p^i(b_i-c_i+2k_i)+\sum_{i=r+1}^{s}p^i(b_i-(p-1)+2k_i-m_{i-1})\end{equation} strictly greater than zero. With other $p^i$-digits ($i>s$) non-negative, we conclude that it cannot be zero hence the dimension of $\PHom_{\FH}(U_{\lfloor m \rfloor_s p}M_b, U_{mp}M_c)$ is zero.
\item Note that the other condition requires $m_{i-1}=0$ on every digit except the $p^s$-digit. Hence from (\ref{POS}) each digit has to be non-negative. Furthermore, the $p^s$-digit is non-negative too, so every digit has to be zero. We can conclude that $b_i=c_i$ for every $i$ by (\ref{POS}) (with $k_i=0$) except when $i=s$, in which case we have $b_s=p-1-m_{s-1}$. Note that the condition 2 in lemma \ref{Holloway} is automatically satisfied by $k_r$ in this case, which is shown by the fact that $c_r\leq p-2$ by definition. Thus, we have $b'_r=p-2-c_r\geq 0=k_r$.
\end{enumerate}
Now we consider the only possible non-trivial extension. The conditions require $c$ such that $c_i=b_i$ for $0\leq i \leq n-1$ except when $i=s$ and $c_s=p-1$ as given by (b), which means that $c=\overline{b}$. Now put $i=s$, tensor the sequence (\ref{ExtMb}) by $U_{\lfloor m \rfloor_s p}$, it induces the triangle \begin{equation*}
 U_{p^s(1+b'_s)}U_{\lfloor m \rfloor_s p}\Omega M_{\overline{b}} \to U_{p^s}U_{\lfloor m \rfloor_s p}\Omega  M_{b'} \to U_{\lfloor m \rfloor_s p}M_b \leadsto. \end{equation*}
The indices of the two $U_i$'s in the middle term add up to the ceiling of $m$ by $s$. Now by assumption, $b_s=p-1-m_{s-1}$, so \begin{equation*} p^s(1+b'_s)+\lfloor m \rfloor_s p=p^sm_{s-1}+(m-m_{s-1}p^{s-1})p=mp. \end{equation*} The previous triangle becomes \begin{equation*} U_{mp}\Omega M_{\overline{b}} \to  U_{\lceil m \rceil^{r+1}p}\Omega M_{b'} \to U_{\lfloor m \rfloor_sp}M_b \leadsto. \end{equation*}  
\end{proof}

\begin{proposition}\label{ExtCor2}
Let $m$ be an integer, $1 \leq m \leq p^{n-1}$. Let $M_c$ be a module in layer $i$ with $i \leq r=r_m$, and $M_b$ in layer $s$ with $s>r$ such that $m_{s-1}+b_s \geq p-1$. Then 
\begin{equation*} 
\PHom_{\FH}(U_{\lceil m \rceil^s p}\Omega M_{b'}, U_{mp}M_c)=0.
\end{equation*}
\end{proposition}
\begin{proof}
Applying our assumption to Theorem \ref{Carlson}, the condition requires
\begin{equation}\label{2}
\lceil m \rceil^s p-mp+\sum_{\substack{i=0\\i \neq l}}^{n-1}p^i(b_i-c_i+2k_i)+p^l(-b_l-c_l+2k_l-2)
\end{equation} to be divisible by $p^n-1$ for an $l$ with $0\leq l \leq r < s$, by condition 1 of \ref{Carlson}. We are going to show that there is no solution, by considering all the possible values. We have \begin{equation*} \lceil m \rceil^s p-mp+p^s(b'_s-(p-1)+2k_s)-p^l(-b_l-c_l+2k_l-2) \geq p^{r+1}+p^s-p^r(p)>0,\end{equation*} so the expression (\ref{2}) must be greater than zero.

Note that $b'_s-m_{s-1} \leq -1$. So the maximum value of the expression is 
\begin{equation*}
p^{s+1}+p^s(b'_s+p-1-m_{s-1})+\sum_{\substack{i=0\\i \neq l,s}}^{n-1}2(p-1)p^i-2p^l<\sum_{i=0}^{n-1}2(p-1)p^i=2p^n-2.
\end{equation*}
The only remaining possibility is that the expression (\ref{2}) is equal to $p^n-1$. Similar to proposition \ref{ExtCor1} we consider the partial sum of the expression up to the $p^l$-digit inclusive. In view of the inequality on $p^l$-digits in (\ref{NEG}), it should lie between $-p^{l+1}$ and $-2$. However, if the whole expression is equal to $p^n-1$ it should have remainder $-1$ modulo $p^{l+1}$, a contradiction.
\end{proof}

\section{Non-trivial autoequivalence of the derived category of $SL_2(q)$}

The non-trivial perverse autoequivalence suggested will be proved in this section. Following that we shall see a few related construction via poset perverse equivalence.

\subsection{Main construction and proof}

We shall approach this by constructing a string of algebras such that \begin{itemize} \item all of these algebras have equivalent derived categories (hence also their associated stable module categories), \item the derived equivalences between successive algebras are elementary perverse, and \item the last one is Morita equivalent to the first one \end{itemize} to give the aforementioned derived autoequivalence (which is a perverse autoequivalence). 

First we define our successive elementary perverse tilts of algebras (see Example \ref{elementary}). We then explore their induced equivalences in their stable module categories to prove our main theorem. 
\begin{definition}\label{PervSeq}
Define inductively a string of algebras $\A_m$, $1 \leq m \leq p^{n-1}$, and a bijection $\beta_m$ of the complete set of non-isomorphic simple $\A_m$-modules, $\S_m$, to the complete set of non-isomorphic simple $\B$-modules, $\S$, by the following. 

First, define $\A_0=\B$ with the identical bijection $\beta_0:=\S \to \S_0$. Suppose $\A_{m-1}$ and $\beta_{m-1}$ is already defined, let $\A_m$ be a symmetric algebra such that $\A_m$ is a perverse tilt from $\A_{m-1}$ with derived equivalence \begin{equation*}
F_m:D^b(\A_{m-1}) \to D^b(\A_m)
\end{equation*} perverse relative to \begin{equation*}
(\beta_{m-1}(0 \subset I_{r_m} \subset \S), \varepsilon:p(0)=1; p(1)=0).
\end{equation*} Such algebras $\A_m$ are symmetric \cite{Rickard2} and defined up to Morita equivalence by Proposition \ref{Prop:Perv}. Now we also define a bijection $\beta_m: \S \to \S_m$ via $\beta_{F_m}\beta_{m-1}$, the composition of the earlier induced bijection and the bijection of simple modules required in the perverse equivalence (Lemma \ref{general}). We also transfer the numbering of simple modules from $\S$ to $\S_m$.

We also define specially $\A_{p^{n-1}}=\A$ and $F:D^b(\B) \to D^b(\A)$ as the composition of $F_m$ for $1 \leq m \leq p^{n-1}$.
\end{definition}

We say \emph{execute step $m$} when we apply functor $F_m$ on the derived categories of $\A_{m-1}$ and $\A_m$. The above construction and our main proposition later can be illustrated by the following diagram.

\begin{equation*}\xymatrix{\B\Mod \ar@{^{(}->}[d] & \dots & \A_{m-1}\Mod \ar@{^{(}->}[d] & \A_m\Mod \ar@{^{(}->}[d] \\
D^b(\B) \ar[rr]^{F_{m-1}\dots F_1} \ar[d] & & D^b(\A_{m-1}) \ar[r]^{F_m} \ar[d] & D^b(\A_m) \ar[d]\\
\B\Pmod \ar[d]^{\isom}  & & \A_{m-1}\Pmod \ar[ll]_{\overline{F_1^{-1}}\dots \overline{F_{m-1}^{-1}}}  & \A_m\Pmod \ar[l]_{\overline{F_m^{-1}}}\\
\FH\Pmod & \hspace{1cm} &  & }\end{equation*}

Referring to the above plan, we study the simple modules of the new algebra $\A_m$ from an inductive approach from the previous algebras. The idea is to describe the image of simple $\A_m$-modules in $\FH\Pmod$. More concretely let $M^m_a$ be the image of simple $\A_m$-modules $T_a$ in the stable module category expressed as $\FH\Pmod$. Now we can describe $M^m_a$ using induction from $M^{m-1}_a$, with the rules introduced in Example \ref{elementary}. 

It turns out the terms and extensions are controllable and the result is being summarized into the proposition below. In the following proposition and lemma when we say a module we mean an $\FH$-module.

\begin{proposition}\label{prop:main}
	Fix a number $m$ between 0 and $p^{n-1}$. The set of all $\FH$-modules $M^m_b$, the correspondents of simple $\A_m$-module $T_b$ in $\FH\Pmod$ for $0 \leq b \leq q-2$, can be partitioned into three sets \begin{equation*}
	J_m \cup K_m \cup L_m
	\end{equation*} such that, depending on parity of $k_s=\lfloor m \rfloor_s/p^s$ (of $m$), \begin{enumerate}
		\item $J_m$ consists of $M^m_b$ in layer $s \leq r_m$. The module $M^m_b \in J_m$ is isomorphic to \begin{equation*}
		\begin{cases} U_{mp}M_b & \text{ if } k_s \text{ is even.} \\ U_{mp}M_{b'} & \text{ if } k_s \text{ is odd.}\end{cases}
		\end{equation*}
		\item $K_m$ consists of $M^m_b$ of layer $s > r_m$, with \begin{equation*}
		\begin{cases} m_{s-1}+b_s \geq p-1 & \text{ if } k_s \text{ is even.} \\ m_{s-1}+b'_s \geq p-1 & \text{ if } k_s \text{ is odd.}\end{cases}
		\end{equation*}  The module $M^m_b \in K_m$ is isomorphic to \begin{equation*}
		\begin{cases} U_{\lceil m \rceil^sp}\Omega M_{b'} & \text{ if } k_s \text{ is even.} \\ U_{\lceil m \rceil^sp}\Omega M_b & \text{ if } k_s \text{ is odd.}\end{cases}
		\end{equation*} 
		\item $L_m$ consists of the remaining modules, that are those with $b$ of layer $s > r_m$ with \begin{equation*}
		\begin{cases} m_{s-1}+b_s < p-1 & \text{ if } k_s \text{ is even.} \\ m_{s-1}+b'_s < p-1 & \text{ if } k_s \text{ is odd.}\end{cases}
		\end{equation*} The module $M^m_b \in L_m$ is isomorphic to \begin{equation*}
		\begin{cases} U_{\lfloor m \rfloor_sp} M_b & \text{ if } k_s \text{ is even.} \\ U_{\lfloor m \rfloor_sp} M_{b'} & \text{ if } k_s \text{ is odd.}\end{cases}
		\end{equation*}
		\end{enumerate}
\end{proposition}
\begin{remark} We can check this is indeed a partition by considering modules in layers. $J_m$ contains every module of layer $\leq r_m$, $K_m$ and $L_m$ further partitioned modules in layers $> r_m$. Note that for $p=2$ the statements are the same disregarding parity of $k_s$ since $b=b'$ and $b_s=b'_s=0$. \end{remark}
We will prove this by induction. It is a two-step approach for each inductive step on $m$. The scheme of the two steps approach is illustrated via the following diagram:
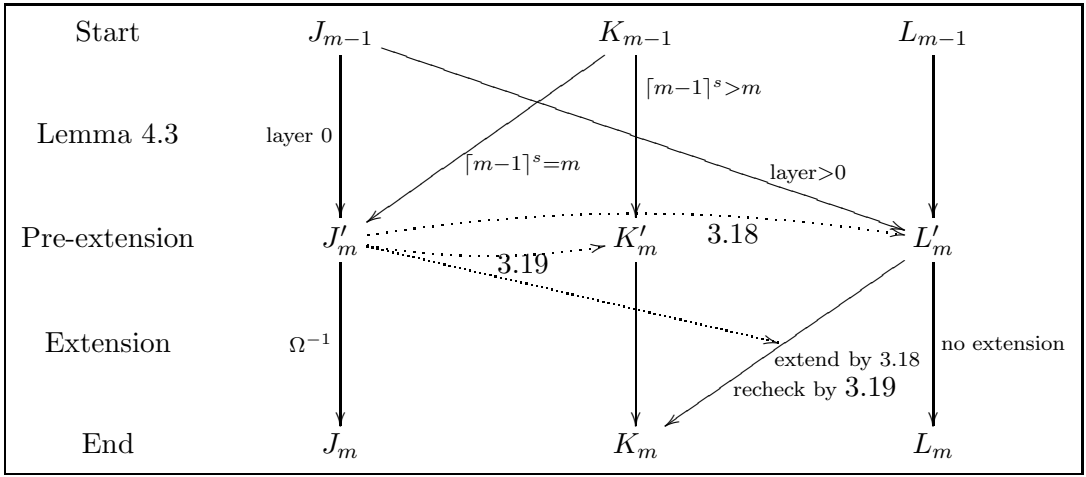
\begin{equation*} \fbox{\xymatrixcolsep{3pc}\xymatrix{\text{Start} & J_{m-1} \ar[dd]_{\text{layer }0} \ar[ddrrrr]^>(.8){\text{layer}>0} & & K_{m-1} \ar[ddll]^>(.6){\lceil m-1 \rceil^s=m} \ar[dd]^>(.3){\lceil m-1 \rceil^s>m} & & L_{m-1} \ar[dd] \\ \text{Lemma \ref{lem:1step}}  & & & & & \\ \text{Pre-extension} & J'_m \ar[dd]_{\Omega^{-1}} \ar@{.>}[drrr]^>(.4){\ref{ExtCor2}} \ar@/_/@{.>}[rr] \ar@/^10pt/@{.>}[rrrr]_>(.7){\ref{ExtCor1}} & & K'_m \ar[dd] & & L'_m \ar[ddll]^{\text{    extend by \ref{ExtCor1}}}^>(.7){\text{    recheck by }\ref{ExtCor2}} \ar[dd]^{\text{no extension}} \\
		\text{Extension} & & & & & \\
		\text{End} & J_m & & K_m & & L_m}}\\ \end{equation*}

First, we have to rewrite the partition. Note that since the partition (based primarily on the parameter $r$) depends on $m$ and subscripts such as $\lceil m \rceil^s$ and $\lfloor m \rfloor_s$ are involved, the induction assumption (from $m-1$) is not in a very usable form for $m$. Hence the first job is to rewrite it into a new partition (with respect to $m$) \begin{equation*} \{M^{m-1}_b \mid 0 \leq b \leq q-2 \}=J'_m \cup K'_m \cup L'_m \end{equation*} such that $J'_m$ corresponds to simple $\A_m$-modules that makes up the foundation of the perverse equivalence $F_m$. The first step is concluded in the following lemma:
\begin{lemma}\label{lem:1step}
	The set of $M^{m-1}_a$, rewriting in the perspective of $m$ and $k_s$ of $m$ (instead of $m-1$), reorganised from the partition in Proposition \ref{prop:main}, is partitioned into \begin{enumerate}
	\item $J'_m$ consists of $M^{m-1}_b$ with the layer of $b \leq r_m$. $M^{m-1}_b$ is isomorphic to \begin{equation*}
	\begin{cases} U_{mp}\Omega M_b & \text{ if } k_s(m) \text{ is even.} \\ U_{mp}\Omega M_{b'} & \text{ if } k_s(m) \text{ is odd.}\end{cases}
	\end{equation*}
	\item $K'_m$ consists of $M^{m-1}_b$ of layer $s > r_m$, with \begin{equation*}
	\begin{cases} m_{s-1}+b_s \geq p-1 & \text{ if } k_s \text{ is even.} \\ m_{s-1}+b'_s \geq p-1 & \text{ if } k_s \text{ is odd.}\end{cases}
	\end{equation*} $M^{m-1}_b$ is isomorphic to \begin{equation*}
	\begin{cases} U_{\lceil m \rceil^sp} \Omega M_{b'} & \text{ if } k_s \text{ is even.} \\ U_{\lceil m \rceil^sp} \Omega M_b & \text{ if } k_s \text{ is odd.}\end{cases}
	\end{equation*} 
	\item $L'_m$ consists of the remaining modules, that is, those with $b$ of layer $s > r_m$ with \begin{equation*}
	\begin{cases} m_{s-1}+b_s < p-1 \text{ or } (m_{s-1}+b_s=p-1 \text { and } m_i=0 \text{ for } 0 \leq i \leq s-2) & \text{ if } k_s \text{ is even.} \\ m_{s-1}+b'_s < p-1 \text{ or } (m_{s-1}+b'_s=p-1 \text { and } m_i=0 \text{ for } 0 \leq i \leq s-2) & \text{ if } k_s \text{ is odd.}\end{cases}
	\end{equation*} $M^{m-1}_b$ is isomorphic to \begin{equation*}
	\begin{cases} U_{\lfloor m \rfloor_sp} M_b  \text{ if } k_s \text{ is even.} \\ U_{\lfloor m \rfloor_sp} M_{b'}  \text{ if } k_s \text{ is odd.}\end{cases}
	\end{equation*} \end{enumerate}
\end{lemma}

Two notes to bear in mind before we proceed. One is we only introduce the parameter $k_s$ for precision of statements. In the following proofs, we shall only argue for the statement starting with $k_s$ even. It is easy to see similarly for odd $k_s$ by exchanging $b \leftrightarrow b'$. Two is we  widely use properties of base-$p$ arithmetic without further comment.

\begin{proof}[Proof of Lemma \ref{lem:1step}]
The idea of this proof is to construct $J'_m$ from the (new) layer constraint and $K'_m$ and $L'_m$ according further to the inherited format as $\FH$-modules.
\begin{itemize}
\item Consider $J'_m$: since $r_m \geq 0$, $M^{m-1}_b$ with $b$ of layer 0 always belongs to $J'_m$, thus their $\FH\Mod$ correspondents, using \ref{PerProj}, can be written as \begin{equation*} U_{(m-1)p}M_b\isom U_{mp}\Omega M_{b'}. \end{equation*} Note that $k_0$ changes parity from $m-1$ to $m$. If $r_m=0$, then we have already found the whole set $J'_m$. If $r_m>0$, we have $r_{m-1}=0$ and $(m-1)_{s-1}=p-1$ for $0 < s \leq r_m$. Hence for $b$ of layer $s$, all of $M^{m-1}_b$ is in $K_{m-1}$. From the condition we have also \begin{equation*}U_{\lceil m-1 \rceil^sp}\Omega M_{b'}\isom U_{mp}\Omega M_{b'} \end{equation*} since $\lceil m-1 \rceil^s=m$ for $0 < s \leq r_m$. Note that $k_s$ is again of different parity since the floor function is differed by $p^s$.

Combining both, now we have grouped into $J'_m$ modules of form $M^{m-1}_b$ of layer $s < r_m$ and isomorphic to $U_{mp}\Omega M_{b'}$ with $k_s$ odd.

\item Now we consider the set $K'_m \cup L'_m$, consisting of $M^{m-1}_b$ with layer of $b$ greater than $r_m$. We have to consider all 3 sources from the $(m-1)^{th}$ statement. First we translate the expressions and conditions to direct terms of $m$ (so to avoid digits of $m-1$ in the subscript). \begin{itemize}
\item First we form the set $K'_m$, which we consider to include all modules of $M^{m-1}_b$ maintaining a ceiling $U$-subscript. They must come only from $K_{m-1}$ since by increasing $m$ neither $J_{m-1}$ nor $L_{m-1}$ can contribute a ceiling $U$-subscript. Note that the formation of the set $J'_m$ takes away all modules of layer $s \leq r_m$ from $K_{m-1}$. They are precisely those with $\lceil m-1 \rceil^s=m$. Note further that $\lceil m-1 \rceil^s$ cannot be $m-1$ since this will force a contradiction with its own condition\footnote{$(m-1)_{s-1}=0$; $b_s\leq p-2$ forces $(m-1)_{s-1}+b_s < p-1$}. Thus every ceiling $U$-subscript from the set $K_{m-1}\backslash J'_m$ has to stay and we take $K'_{m-1}=K_{m-1}\backslash J'_m$. Thus modules in $K'_m$ have the form \begin{equation}\label{Pre2a} U_{\lceil m \rceil^sp}\Omega M_{b'}\qquad \text{ if }m_{s-1}+b_s \geq p-1.\end{equation} 
The last condition can be switched from $m-1$ to $m$ directly by properties of natural numbers: The only case where they differs is $m_{s-1} < (m-1)_s-1$, which is equivalent to $p^s$ divides $m$ and those modules are moved away from $K_{m-1}$ to $J'_m$. 
\item Now we consider the remaining set $L'_m$, which consists of modules with a floor $U$-subscript. Since $m$ is only increased by 1, modules in $K_{m-1}$ will not be rewritten\footnote{$\lceil m-1 \rceil^s$ cannot be $m-1$ and if $\lceil m-1 \rceil^s=m$ it has been assigned to $J'_m$.} into $L'_m$. First we consider modules coming from $L_{m-1}$. To rewrite the $U$-subscript from $m-1$ into $m$ we need to consider the case $\lfloor m-1 \rfloor_s \neq \lfloor m \rfloor_s$. However this only happens when $m$ is divisible by $p^s$ hence $(m-1)_{s-1}=p-1$ thus this case is not included in $L_{m-1}$. We can safely change the subscript $\lfloor m-1 \rfloor_s$ from modules in $L_{m-1}$ to $\lfloor m \rfloor_s$. Secondly we consider modules coming from $J_{m-1}$. If there are such modules then $r_{m-1}>s>0$ and hence $r_m=0$. Their expression $U_{(m-1)p}M_b$, translate to $m$ is equal to \begin{equation*} U_{\lfloor m \rfloor_sp} M_b,\text{ with }(m-1)_{s-1}+b_s = b_s < p-1.\end{equation*}
This expression coincides with those coming from $L_{m-1}$. In conclusion we have the modules in $L'_m$ isomorphic to \begin{equation}\label{Pre2b} U_{\lfloor m \rfloor_sp} M_b\text{ with }(m-1)_{s-1}+b_s < p-1.\end{equation} 
Translating $(m-1)_{s-1}+b_s < p-1$ in terms of $m$ (instead of relying on $m-1$) cause it to split into two: \begin{equation}\label{Cond2b}m_{s-1}+b_s < p-1\qquad \text{ or }\qquad m_{s-1}+b_s=p-1\text{ and }m_{s-2}=...=m_0=0. \end{equation}
\end{itemize}
\end{itemize}
It is easy to check that $k_s$ does not change when $r_m < s$. By these arguments we have successfully re-partitioned $M^{m-1}_b$ as indicated in the lemma.
\end{proof}
Of course, the re-partition in Lemma \ref{lem:1step} is tailored such that we can apply the correspondence in perverse equivalence in a fairly convenient manner. The set $J'_m$ corresponds to the foundation of the perverse equivalence $F_m$. $K'_m$ and $L'_m$ is grouped by expression and needed to check for extensions.
\begin{proof}[Proof of Proposition \ref{prop:main}]
Now we start the main proof by considering the algebra $\A_0=\B$. The module $M^0_b$ is simply $M_b$ and the set of such modules is partitioned as $J_0\cup K_0 \cup L_0 = \S \cup \emptyset \cup \emptyset$ respectively. Thus the statement is true at $m=0$ (assuming large enough $r_0$), which allows us to start the induction for $m\geq 1$. Now assume the statement is true for an $m-1$. Let $S_b \in \S_{m-1}$ be a simple $\A_{m-1}$-module corresponding to a simple $\A_m$-module $T_b \in \S_m$. The induction step requires us to find the image of $T_b$ in $\FH\Pmod$ via $M^{m-1}_b$. Using the stable category equivalent of Example 1.79, we have: \begin{equation*}
\overline{F_m^{-1}}(T)=\begin{cases} \Omega^{-1}S &\text{ if } \beta_m(T) \in I_{r_m}  \\ S_{\S'} &\text{ otherwise.}\end{cases}
\end{equation*}
\begin{itemize}
\item For $\beta_m(T_b) \in I_{r_m}$ is equivalent to say $b$ is of layer between 0 and $r_m$ inclusive. The corresponding $\FH$-module of $S_b$ is in the set $J'_m$. Then $T_b$ corresponds to \begin{equation*}
\Omega^{-1}(U_{mp} \Omega M_b)\isom U_{mp} M_b
\end{equation*} in $\FH\Pmod$ which proves the statement for the set $J_m$.
\item Before we argue on the modules that require us to check extensions, note that $J'_m$ includes all modules of layer $\leq r_m$, hence we disregard $b$ or $b'$ when considering extensions. Further see that we can replace $b$ by $b'$ in Proposition \ref{ExtCor1} and \ref{ExtCor2} thus they always apply to $L'_m$ and $K'_m$ respective regardless of the parity of $k_s$. 
\item For the remaining $T_b$'s such that $\beta_m(T_b) \notin I_{r_m}$, we need to find the universal extension of $S_b$ by the set $\S'$, where $\beta_{m-1}(\S')=I_{r_m}$, which is equivalent to consider the universal extension of any element in $K'_m \cup L'_m$ by $J'_m$ in $\FH\Pmod$. For any module in $L'_m$ we first use Proposition \ref{ExtCor1} to check the required extension. Only modules corresponding to (\ref{Pre2b}) satisfying the last condition in (\ref{Cond2b}) have one-dimensional extensions. The module after extension is isomorphic in $\FH\Pmod$ to \begin{equation}\label{Aft2b} U_{\lceil m \rceil^sp}\Omega M_{b'}\end{equation} according to triangle (\ref{ExtMb}). The rest of the modules in $L'_m$ satisfying the first condition in (\ref{Cond2b}) are not extendible, so have their corresponding expression remains the same. With this, we show Proposition 3.2 is true for $L_m$. Now note that (\ref{Aft2b}) has exactly the same expression as those in $K'_m$, see (\ref{Pre2a}). Their respective condition can be joined up perfectly as $m_{s-1}+b_s \geq p-1$, exactly what the $m^{th}$ proposition statement and Proposition \ref{ExtCor2} required. Now Proposition \ref{ExtCor2} has shown that all modules have no more available extensions. This has formed the required $K_m$ part of the partition of the induction statement. 
\end{itemize}
Thus we have successfully show that the statement for $m$ is true for all three parts of the partition. Hence it is true by induction up to $\A_{p^{n-1}}$.
\end{proof}

\begin{corollary}\label{Res}
The image of simple $\A$-modules $S_a$ in $\FH\Pmod$ is $U_{q} M_{a'}\isom U_1M_{a'}$.
\end{corollary}
\begin{proof}
By the induction statement all modules correspond to $J_{p^{n-1}}$ for $\A$. So all the simple modules correspond to $\FH$-modules $U_{q}M_a'$, since $k_s=p^{n-1-s}$ is an odd number for odd primes. When $p=2$, $a=a'$ so it makes no discrimination.
\end{proof}

Now our desired result is imminent. 
\begin{theorem}\label{main}
There exists a non-trivial perverse autoequivalence of direct sum of all full defect blocks $\B$ of $\FG$ exchanging blocks. \end{theorem}

\begin{proof}
This theorem is immediate after we show that $\A$ is Morita equivalent with $\B$. Now consider the image of simple $\A$-module in its stable category,
\begin{equation*}\xymatrix{\A\Pmod \ar@{=}[r] & \FH\Pmod \ar[rr]^{U_{-1}\tensor -}_= & \hspace{1cm} & \FH\Pmod \ar[r]^{\mathrm{Ind}}_= & \B\Pmod \\
T_a \ar@{=}[r] & U_1M_{a'} \ar@{|->}[rr] & & M_{a'} \ar@{|->}[r] & S_{a'}}\end{equation*}

Note that both the functor $U_{-1} \tensor -$ and induction (the functor is ${}_{\B}\B_{\FH} \tensor_{\FH} - $) are stable equivalences of Morita type, and their composition maps simple $\A$-modules to simple $\B$-modules. Thus using Theorem \ref{Trick} we conclude that $\A$ is Morita equivalent to $\B$.

Putting $\B_0$ through the perverse tilts yields one of the blocks in $\B$. Similar to the proof of theorem above, note that corollary \ref{Res} indicate that the even-numbered simple $\B$-modules have stable images that are odd-numbered simple $\B$-modules. Thus the block obtained by tilting $\B_0$ using \ref{PervSeq} should be $\B_1$. Lastly all elementary perverse equivalence can be refined as one since their filtration is nested, thus refineability applies.
\end{proof}

\begin{remark}
	Since we proved $\A$ and $\B$ are Morita equivalent, $F:D^b(\B) \to D^b(\A)$ will be consider as an autoequivalence thereafter. The autoequivalence is perverse with respect to $(I_\bullet, I_\bullet, \pi)$ and perversity function $\pi: i \mapsto -p^{n-1-i}$. The correspondence of simple is $b \leftrightarrow b'$.
\end{remark}

This construction can also be seen as a generalisation of Morita equivalence for the two full defect blocks in $\F SL_2(p)$. For odd prime $p$, the two blocks is known to be Morita equivalent as Brauer Tree algebras. 

\subsection{Frobenius Twist and an invariant construction}

The equivalence $F$ we have demonstrated so far is not Frobenius invariant. This can be easily observed, since the Frobenius automorphism on $H$-modules maps the $\FH$-module $U_1$ to $U_p$, thus hinting the construction is twisted to another self-equivalence. However, we can find a self-equivalence which is Frobenius invariant by introducing a composition of functors generated using the Frobenius automorphism on $\FG$-modules.

Consider the fact that perverse equivalence is defined up to Morita equivalence, given a perverse equivalence, one can compose/conjugate it with some known functor inducing Morita equivalence to obtain further results. If we use Frobenius automorphism on the module category of $\B$-modules to conjugate our construction we have the following.

\begin{definition}\label{def:Frob conj}
	Let $F$ be the autoequivalence on $D^b(\B)$ defined in Theorem \ref{main}. The Frobenius automorphism $\sigma$ of $\B$ induces an automorphism on $D^b(\B)$, which we also call $\sigma$. Define $F^\sigma$, \emph{the Frobenius conjugate} of $F$, to be the functor $\sigma F \sigma^{-1}$. Define $F^{\sigma^i}=\sigma^i F \sigma^{-i}$ similarly for $0 \leq i \leq n-1$. 
	
	Let $E$ be the composition of functors $F^{\sigma^i}$ for $i$ from $0$ to $n-1$ in that order, i.e. \begin{equation*} E=F^{\sigma^{n-1}}\circ F^{\sigma^{n-2}} \circ ... \circ F^{\sigma} \circ F. \end{equation*} \end{definition} 

Since $\sigma$ restricts to an equivalence on $\B\Mod$, the new functors $F^{\sigma^i}$ introduced are all perverse equivalences. However, they are associated with different filtrations and perversity functions. More precisely, $F^{\sigma^i}:D^b(\B) \to D^b(\B)$ is perverse relative to \begin{equation*}
(I_\bullet^{\sigma^i}, I_\bullet^{\sigma^i}, \pi:j \mapsto -p^{n-1-j}).
\end{equation*}

We are going to compose the derived equivalences generated by Frobenius automorphism introduced above. For that purpose, we consider $F$ in the sense of poset perverse equivalence. Using the information of the proof of Proposition \ref{prop:main}, we have

\begin{proposition} The automorphism $F$ is perverse relative to $(\S, \prec, \pi)$, where $\prec$ is the relation \begin{equation*} a\prec b \text{ if } a=\overline{b}. \end{equation*} \end{proposition}
\begin{proof}
The extension happened in the proof of Proposition \ref{prop:main} is a module $S_a$ extended by $S_{\overline{a}}$(c.f. definition \ref{Corr}). Consider all modules have been extended by this way we are done.\end{proof} 

We define another partial order on $\B$ below. It is a refinement of the earlier partial order. We can compose $F$ and its Frobenius twists under this partial order.

\begin{definition}\label{def:POofB}
Define a partial order $\prec_{\B}$ on the set $\S$: $S_a \prec_{\B} S_b$ if, for their $n$-digit base-$p$ representations, $a$ has more digits of $p-1$ and $a_i=p-1$ whenever $b_i=p-1$. \end{definition}

\begin{proposition} The equivalence $F^{\sigma^i}$ is perverse relative to $(\S_a, \prec_{\B}, \omega^{\sigma^i})$. \end{proposition}

\begin{proof}
	What we need is $\prec_{\B}$ compatible with $F^{\sigma^i}$. Since $\prec_{\B}$ is Frobenius invariant, it equates to be compatible with $F$, but this is obvious. 
\end{proof}

With the partial order being set up we focus on the perversity.

\begin{definition}
	Define a map \begin{equation*}
	\omega: \S \to \Z
	\end{equation*} on simple $\B$-modules such that $\omega$ is the composition of the layer map and $\pi$. We further define \begin{equation*}
	\omega^{\sigma^i}: \S \to \Z
	\end{equation*} to be $\omega$ pre-composed by $\sigma^i$ on the set $\S$.
\end{definition} 
 
Then we have the following:
\begin{theorem} The functor $E$ defined in \ref{def:Frob conj} is perverse relative to \begin{equation*}
	(\S_a, \prec_{\B}, \sum_{i=0}^{n-1}\omega^{\sigma^i}).
	\end{equation*}Furthermore, $E$ is Frobenius invariant.
\end{theorem}
\begin{proof} Now Frobenius conjugation on $E$ yields a functor \begin{equation*}
	F^{\sigma^{n}} \circ F^{\sigma^{n-1}} \circ ... \circ F^{\sigma}
	\end{equation*} which is a cyclic permutation of $E$ (note $F^{\sigma^n}=F$). Since by the first assertion they are all compatible with the same partial order, they commute by Proposition \ref{Prop:PosPerv}. Thus after rearrangement we get back $E$. The new perversity function is just the sum of all perversity functions from $F$ and its Frobenius twists. To see that the sum is Frobenius invariant, observe that applying $\sigma$ rotates the sum.\end{proof}

We are going to complete this subsection by showing that we can define a filtration on $\B$ for $E$ which is Frobenius invariant. By doing this we further see we can group simple $\B$-modules using partitions of $n$ and in our case, the perversity function on $S_a$ only depends on the partition.

\begin{definition}\label{def:FrobInvFil}
	Let $S_a \in \S$ for an integer $a$, $0 \leq a \leq p^n-2$. Assign a partition $\lambda_a \dashv n$ to $S_a$ using the following steps. \begin{enumerate}
		\item Denote by $(Z_a)_i$ the layer of $p^ia$ modulo $p^n-1$ 
		\item Let $\lambda=(\lambda_0,\dots,\lambda_{n-1})$ be a $n$-tuple, where $\lambda_j$ is the number of times $n-1-j$ occurred in $(Z_a)_i$ for $0 \leq i \leq n-1$. 
	\end{enumerate}
\end{definition}
	We can show that $\lambda$ is indeed a partition, i.e. $\lambda_0\geq \lambda_1 \geq \dots \geq \lambda_{n-1}$, by the following lemma.
	\begin{lemma}
		Let $a$ be in layer $i<n-1$, then $p^a$ modulo $p^n-1$ is in layer $j+1$.
	\end{lemma}
	\begin{proof} If $a$ is in layer $j<n-1$ then we have \begin{equation*}
		p^n-p^{j+1} \leq a \leq p^n-p^j-1.
		\end{equation*}
		Multiplying by $p$ we have 
		\begin{equation*}
		p^{n+1}-p^{j+2} \leq pa \leq p^{n+1}-p^{j+1}-p.
		\end{equation*} 
	Since $i+2<n$, $p^{n+1}=p^n+p-1$ modulo $p$, we have 
	\begin{equation*}
	p^n-p^{j+2}+(p-1) \leq pa \leq p^n-p^{j+1}-1
	\end{equation*}  
	showing $pa$ is of layer $j+1$.
\end{proof}
	This shows for any $(Z_a)_i=j<n-1$ we have $(Z_a)_{i+1}=j+1$, hence $\lambda$ is a partition.	
	
By the above we have defined a map \begin{equation*}
	l: \S \to \{\text{Partitions of }n\}
	\end{equation*} and we defined another function \begin{equation*} \begin{split}
	\phi':\{\text{Partitions of }n\} &\to \Z\\
	\lambda=(\lambda_0,\dots,\lambda_{n-1}) &\mapsto \sum_{i=0}^{n-1}-\lambda_ip^{n-1-i}
	\end{split} \end{equation*}

\begin{proposition} For the definitions above we have the following: \begin{enumerate}
		\item The partial order $\prec_{\B}$ we defined in Definition \ref{def:POofB} is collapsed by $l$ into the reverse dominance order of partitions.
		\item The map $\phi'$ is injective. 
		\item $\phi'(a) < \phi'(b)$ when $\lambda_a < \lambda_b$ in lexicographical order\footnote{i.e. The partition $(n)$ is the last (greatest) term and $(1^n)$ is the first (smallest) term.}.
	\end{enumerate} \end{proposition}
	\begin{proof} Let $a\prec_{\B}b$ be two integers and $a_i=p-1\neq b_i$ for some $i$. Now notice that $(Z_a)_{i+1}<(Z_b)_{i+1}=n-1$. Then by previous lemma we have 1. For two partitions $\lambda$ and $\lambda'$, let $j$ be the greatest integer such that $\lambda_j\neq\lambda_j'$. Then $\phi'(a)-\phi'(b) \neq 0 \mod{p^j}$ thus we have 2. The last one is just a check on the sum.
	\end{proof}

\begin{example} Consider $77$ in $p^n=3^6$, its $6$-digit base-$3$ representation is $(2,1,2,2,0,0)$, then we have \begin{equation*} (Z_a)=(5,5,3,4,5,4) \end{equation*} hence $77$ correspond to a $6$-partition $(3,2,1)$. The function $q' \circ l$ maps $S_{77}$ to $3(-1)+2(-3)+1(-9)=-18$. \end{example}

Now we apply the new filtration on simple $\B$-modules and the perversity function on the filtration. 
\begin{definition}
	Define $I$ to be a perverse $(q-1)$-data by the following: First, assign any integer $a \in \{0, \dots, q-2\}$ to the set $J_{\lambda}$, where $\lambda \dashv n$ is the partition representing $a$.
	Then we order partitions using lexicographical order, $\prec$ and let \begin{equation*}
	I_{\lambda}=\bigcup_{\kappa\prec\lambda}J_{\kappa}.
	\end{equation*} and set \begin{equation*}
	I_\bullet=(\emptyset \subset I_{(1^n)} \subset I_{(1^{n-2}2)} \subset \dots \subset I_{(n)}=\S).
	\end{equation*}
	
	The perversity function is taken as $\phi'$ defined in Definition \ref{def:FrobInvFil}. 
	
	Define $E':D^b(\B) \to D^b({}^I\B)$ to be a derived equivalence generated by the perverse $q-1$-data on (naturally $q-1$-indexed) $\B$.
\end{definition}
\begin{remark}
	We have indexed the filtration by partitions, but this does not affect anything for using the idea of perverse data.
\end{remark}

Lastly, we conclude all the discussion as a theorem.

\begin{theorem}
	Let $E':D^b(\B) \to D^b({}^I\B)$ be the derived equivalence defined using the above data, perverse relative to \begin{equation*}
	(\emptyset \subset I_{(n)} \subset \dots I_{(1^n)}=\S, \phi')
	\end{equation*} We have the equivalence $E=F^{\sigma^{n-1}}\dots F:D^b(\B) \to D^b(\B)$ is compatible with $E'$. Therefore \begin{enumerate}
		\item ${}^I\B$ is Morita equivalent to $\B$. \item $E'$ is Frobenius invariant.
		\item $E'$ is of increasing perversity (c.f. \cite{Craven}).
	\end{enumerate} 
\end{theorem}
\begin{proof}
	All these can be achieved as long as we show $E'$ is actually $E$, or showing $E'$ is compatible with $E$. First we show the filtration of $E'$ is a refinement of the poset order in $E$. Recall that every extension in $E$ (which comes from various Frobenius twists of $F$) has more $p-1$'s in its base-$p$ expression (since it is a composition of $F$ and its Frobenius twists) and thus maps to a lower filtrate in $E'$. So the partial order defined in $E$ is compatible with the filtration of $E'$. Now it remains to check the perversity function of the two equivalences is the same on all simple modules $S_a$. This is not difficult to see since \begin{equation*}
	\sum_{i=0}^{n-1}\omega^{\sigma^i}=\sum_{i=0}^{n-1}-p^{n-1-(Z_a)_i}=\sum_{i=0}^{n-1}-\lambda_ip^{n-1-i}
	\end{equation*} by rearranging according to $p$-powers.
		
Thus we have checked their perversity is equal on all simple modules and thus $E$ and $E'$ is compatible. The fact that it is of increasing perversity comes from $E'$'s perversity function $\phi'$.
\end{proof}

\section{Examples}

\begin{notation} For clarity and convenience, we suppress the $U$'s in Loewy structures yet to appear. We have also taken modulo $p^n-1$ on the subscript when possible because they are isomorphic (see section 2). Note that, from now on, a 0 is representing a trivial module $U_0$ and we explicitly say zero for a zero module.\end{notation}

\subsection{$SL_2(4)$}

The first non-trivial equivalence introduced by our construction involves $G=SL_2(2^2)$. It is a somewhat special case, since $p=2$ and thus we have only one non-semisimple block instead of two. However it is good enough for us to demonstrate the construction while decoding some intriguing features. The group $SL_2(4)$ is isomorphic to the alternating group $A_5$ of five elements. A Sylow 2-subgroup of $A_5$ is the Klein 4-group, and can be chosen so that its normaliser in $A_5$ is $A_4$. This example has been studied by many (e.g. \cite{Rickard2}), since its representation type is tame.

There are 4 non-isomorphic simple $\FA_5$-modules: The trivial module $k=S_0$, two modules $V=S_1$, $W=S_2$ which are two-dimensional and the (4-dimensional) Steinberg module $S_3$. Their corresponding indecomposable projective covers have Loewy series as follows:
\begin{equation*} P_k=\begin{matrix} & k & \\ V & & W \\ k & & k \\ W & & V \\ & k & \end{matrix} \qquad \qquad P_V=\begin{matrix} V\\k\\W\\k\\V \end{matrix} \qquad \qquad P_W=\begin{matrix} W\\k\\V\\k\\W \end{matrix} \qquad \qquad St. \end{equation*}
We shall ignore $S_3$ as it is lies in the Steinberg block $St$, a simple block of $\FG$. The remaining simple modules form a full defect block, the principal block $\B_0$. There are 3 non-isomorphic simple $\FA_4$-module. The trivial module $k=U_0$, $U_1$ and $U_2$, all one-dimensional. Their corresponding indecomposable projective covers have Loewy series (with $U$ suppressed) as follows: \begin{equation*} Q_0=\begin{matrix} & 0 & \\ 1 & & 2 \\ & 0 & \end{matrix} \qquad \qquad Q_1=\begin{matrix} & 1 & \\ 2 & & 0 \\ & 1 & \end{matrix} \qquad \qquad Q_2=\begin{matrix} & 2 & \\ 0 & & 1 \\ & 2 & \end{matrix}. \end{equation*} The restrictions of simple $A_5$-modules to $A_4$ are given by \begin{equation*} k\hspace{-.4em}\downarrow_{A_4}=0 \qquad V\hspace{-.4em}\downarrow_{A_4}=M_1=\begin{matrix}1\\2\end{matrix} \qquad W\hspace{-.4em}\downarrow_{A_4}=M_2=\begin{matrix}2\\1\end{matrix} \qquad  St\hspace{-.4em}\downarrow_{A_4}=M_3=\begin{matrix} & 0 & \\ 1 & & 2 \\ & 0 &\end{matrix} =Q_0. \end{equation*} Adopting the terminology of 'layers' from Definition \ref{Layer}, $W$ is in layer 0 and the remaining simple $\FA_5$-modules are in layer 1. The only non-trivial step of the procedure carried out in Definition \ref{PervSeq} is at $m=1$, for which the one-sided tilting complex of $\B_0$-modules is \begin{equation*} P_k \oplus P_V \oplus P_k \xrightarrow{\alpha} P_W \end{equation*} where $\alpha:P_k \to P_W$ is a presentation of the simple module $W$. The following is a table of the images of simple $\A_i$-modules in $D^b(\B_0)$, and in $\FA_4\Pmod$. 
\begin{center} \begin{tabular}{|c|c|c|c||c|c|c|} \hline
\backslashbox{in algebra}{simples numbered}& 0 & 1 & 2 & 0 & 1 & 2 \\ \hline
$\B=\A_0$ & $k$ & $V$ & $W$ & $0$ & $\begin{matrix}& 1\\2 &\end{matrix}$ & $\begin{matrix}2&\\&1\end{matrix}$\\ \hline
$\A_1$ & $\begin{matrix}k\\W\end{matrix}$ & $V$ & $W[1]$ & $\begin{matrix}2 & & 0\\ & 1 &\end{matrix}$ & $\begin{matrix}&1\\2&\end{matrix}$ & $\begin{matrix}1&\\&0\end{matrix}$\\ \hline
$\A_2=\A$ & $\begin{matrix}k\\W\end{matrix}[1]$ & $V[1]$ & $W[2]$ & $1$ & $\begin{matrix}&2\\0&\end{matrix}$ & $\begin{matrix}0&\\&2\end{matrix}$\\ \hline
in category: & \multicolumn{3}{|c||}{$D^b(\B_0)$} & \multicolumn{3}{|c|}{$\FA_4\Pmod$} \\ \hline
\end{tabular} \end{center}

The benefit of writing in $\FA_4\Pmod$ may not be obvious in this example, but will be immense when generalised in either $p$ or $n$.

We list out the composition of $F$ and $F^{\sigma}$ in the following table
\begin{center} \begin{tabular}{|c|c|c|c||c|c|c|} \hline
		\backslashbox{in algebra}{simples numbered}& 0 & 1 & 2 & 0 & 1 & 2 \\ \hline
		$\B$ & $k$ & $V$ & $W$ & $0$ & $\begin{matrix}& 1\\2 &\end{matrix}$ & $\begin{matrix}2&\\&1\end{matrix}$\\ \hline
		$F(\B)$ & $\begin{matrix}k\\W\end{matrix}[1]$ & $V[1]$ & $W[2]$ & $1$ & $\begin{matrix}&2\\0&\end{matrix}$ & $\begin{matrix}0&\\&2\end{matrix}$\\ \hline
		$F^{\sigma}F(\B)$ & $\begin{matrix}& k &\\W & & V\end{matrix}[2]$ & $V[3]$ & $W[3]$ & $0$ & $\begin{matrix}&1\\2&\end{matrix}$ & $\begin{matrix}2&\\&1\end{matrix}$\\ \hline
		in category: & \multicolumn{3}{|c||}{$D^b(\B_0)$} & \multicolumn{3}{|c|}{$\FA_4\Pmod$} \\ \hline
	\end{tabular} \end{center}

\subsection{Further examples}

From the earlier tables in $SL(2,4)$ we see how information of the local stable category governs the extensions of global modules. Now we list the local stable correspondence and the extensions which occur in each elementary perverse step in our proof when $G=SL_2(3^2)$, $SL_2(5^2)$ and $SL_2(3^3)$. In the tables for $SL_2(3^2)$ we also give the corresponding simple collection in each intermediate algebra. Local stable equivalent tables are enough to remake the process in derived category.


\begin{table}
\caption{Case for $SL_2(3^2)$}
\begin{subtable}{1\textwidth}
\centering
\subcaption{The local stable category equivalent for $B_0(SL_2(3^2))$}
\renewcommand{\arraystretch}{1.3}
\begin{tabular}{|l||c|c|c||c|} \hline
$B_0(3^2)$ & \multicolumn{3}{|c||}{Layer 1} & Layer 0 \\ \hline
$\A_0$ & $M_0$ & $M_2$ & $M_4$ & $M_6=U_3\Omega M_7$  \\ \hline
step 1 &  &  & $M_6=U_3\Omega M_7$ & $\Omega^{-1}$ \\ \hline
$\A_1$ & $M_0$ & $M_2$ & $U_1 \Omega M_1$ & $U_3M_7= U_6 \Omega M_6$ \\ \hline
step 2 & $U_3M_7=U_6\Omega M_6$ & $M_8=0$ &  & $\Omega^{-1}$ \\ \hline
$\A_2$ & $U_1\Omega  M_3$ & $U_1 \Omega M_5$ & $ U_1\Omega M_1$ & $U_6M_6=U_1 \Omega M_7$ \\ \hline
step 3 &  $\Omega^{-1}$ & $\Omega^{-1}$ & $\Omega^{-1}$ & $\Omega^{-1}$ \\ \hline
$\A_3$ & $U_1 M_3$ & $U_1 M_5$ & $U_1 M_1$ & $U_1 M_7$ \\ \hline
\end{tabular}
\vspace{11pt}
\end{subtable}
\begin{subtable}{1\textwidth}
\centering
\subcaption{The construction as in $D^b(SL_2(3^2))$ - the corresponding complex for simples in $B_0(\A_m)$}
\renewcommand{\arraystretch}{1.3}
\begin{tabular}{|l||c|c|c||c|} \hline
$B_0(3^2)$ & \multicolumn{3}{|c||}{Layer 1} & Layer 0 \\ \hline
 &  $S_0$ & $S_2$ & $S_4$ & $S_6$ \\ \hline
$\A_0$ & $S_0$ & $S_2$ & $S_4$ & $S_6$  \\ \hline 
step 1 &  &  & map to $S_6[1]$ and take co-cone & [1] \\ \hline
$\A_1$ & $S_0$ & $S_2$ & $\begin{matrix} S_4 \\[-0.5em] S_6 \end{matrix}$ & $S_6[1]$ \\ \hline
step 2 & map to $S_6[2]$ and take co-cone & * &  & [1] \\ \hline
$\A_2$ & $\mathrm{co\hyp cone}(S_0 \to S_6[2])$ & $S_2$ & $\begin{matrix} S_4 \\[-0.5em] S_6 \end{matrix}$ & $S_6[2]$ \\ \hline
$\A_3$ & $\mathrm{cone}(S_0 \to S_6[2])$ & $S_2[1]$ & $\begin{matrix} S_4 \\[-0.5em] S_6 \end{matrix}[1]$ & $S_6[3]$ \\ \hline
\multicolumn{5}{l}{*This step take its map to 0 and obtain its co-cone, which obtains an isomorphic module.}
\label{CPX}
\end{tabular}
\end{subtable}
\end{table}

\begin{landscape} 
\begin{table}
\caption{Exchanging blocks of $SL_2(5^2)$ under the construction, local stable category equivalent}
\centering 
\renewcommand{\arraystretch}{1.25}
\begin{tabular}{|l||c|c|c|c|c|c|c|c|c|c||c|c|c|}
\multicolumn{7}{l}{Shorthand: $U_i \Omega^j M_a$ written as ${}^j_i M_a$} & \multicolumn{6}{r}{$U$ subscripts modulo $5^2-1=24$} \\ \hline
$B_0(5^2)$ & \multicolumn{10}{|c||}{Layer 1} & \multicolumn{2}{|c|}{Layer 0} \\ \hline
 & $S_0$ & $S_2$ & $S_4$ & $S_6$ & $S_8$ & $S_{10}$ & $S_{12}$ & $S_{14}$ & $S_{16}$ & $S_{18}$ & $S_{20}$ & $S_{22}$ \\ \hline
$\A_0$ & $M_{(0,0)}$ & $M_{(2,0)}$ & $M_{(4,0)}$ & $M_{(1,1)}$ & $M_{(3,1)}$ & $M_{(0,2)}$ & $M_{(2,2)}$ & $M_{(4,2)}$ & $M_{(1,3)}$ & $M_{(3,3)}$ & ${}^1_5 M_{(3,4)}$ & ${}^1_5 M_{(1,4)}$ \\ \hline
step 1 &  &  &  &  &  &  &  &  & ${}^1_5 M_{(1,4)}$ & ${}^1_5 M_{(3,4)}$ & $\Omega^{-1}$ & $\Omega^{-1}$ \\ \hline
$\A_1$ & $M_{(0,0)}$ & $M_{(2,0)}$ & $M_{(4,0)}$ & $M_{(1,1)}$ & $M_{(3,1)}$ & $M_{(0,2)}$ & $M_{(2,2)}$ & $M_{(4,2)}$ & ${}^1_1 M_{(1,0)}$ & ${}^1_1 M_{(3,0)}$ & ${}^1_{10} M_{(0,4)}$ & ${}^1_{10} M_{(2,4)}$ \\ \hline
step 2 &  &  &  &  &  & ${}^1_{10} M_{(0,4)}$ & ${}^1_{10} M_{(2,4)}$ & $M_{(4,4)}=0$ & & & $\Omega^{-1}$ & $\Omega^{-1}$ \\ \hline
$A_2$ & $M_{(0,0)}$ & $M_{(2,0)}$ & $M_{(4,0)}$ & $M_{(1,1)}$ & $M_{(3,1)}$ & ${}^1_1 M_{(0,1)}$ & ${}^1_1 M_{(2,1)}$ & ${}^1_1 M_{(4,1)}$ & ${}^1_1 M_{(1,0)}$ & ${}^1_1 M_{(3,0)}$ & ${}^1_{15} M_{(3,4)}$ & ${}^1_{15} M_{(1,4)}$\\ \hline
step 3 &  &  &  & ${}^1_{15} M_{(1,4)}$ & ${}^1_{15} M_{(3,4)}$ &  &  &  &  &  & $\Omega^{-1}$ & $\Omega^{-1}$ \\ \hline
$\A_3$ & $M_{(0,0)}$ & $M_{(2,0)}$ & $M_{(4,0)}$ & ${}^1_1 M_{(1,2)}$ & ${}^1_1 M_{(3,2)}$ & ${}^1_1 M_{(0,1)}$ & ${}^1_1 M_{(2,1)}$ & ${}^1_1 M_{(4,1)}$ & ${}^1_1 M_{(1,0)}$ & ${}^1_1 M_{(3,0)}$ & ${}^1_{20} M_{(0,4)}$ & ${}^1_{20} M_{(2,4)}$ \\ \hline
step 4 & ${}^1_{20} M_{(0,4)}$ & ${}^1_{20} M_{(2,4)}$ & $M_{(4,4)}=0$ &  &  &  &  &  &  &  & $\Omega^{-1}$ & $\Omega^{-1}$ \\ \hline
$\A_4$ & ${}^1_1 M_{(0,3)}$ & ${}^1_1 M_{(2,3)}$ & ${}^1_1 M_{(4,3)}$ & ${}^1_1 M_{(1,2)}$ & ${}^1_1 M_{(3,2)}$ & ${}^1_1 M_{(0,1)}$ & ${}^1_1 M_{(2,1)}$ & ${}^1_1 M_{(4,1)}$ & ${}^1_1 M_{(1,0)}$ & ${}^1_1 M_{(3,0)}$ & ${}^1_1 M_{(3,4)}$ & ${}^1_1 M_{(1,4)}$ \\ \hline
$\A_5$ & ${}_1M_{15}$ & ${}_1M_{17}$ & ${}_1M_{19}$ & ${}_1M_{11}$ & ${}_1M_{13}$ & ${}_1M_{5}$ & ${}_1M_{7}$ & ${}_1M_{9}$ & ${}_1M_{1}$ & ${}_1M_{3}$ & ${}_1M_{23}$ & ${}_1M_{21}$ \\ \hline \hline

$B_1(5^2)$ & $S_1$ & $S_3$ & $S_5$ & $S_7$ & $S_9$ & $S_{11}$ & $S_{13}$ & $S_{15}$ & $S_{17}$ & $S_{19}$ & $S_{21}$ & $S_{23}$ \\ \hline
$\A_0$ & $M_{(1,0)}$ & $M_{(3,0)}$ & $M_{(0,1)}$ & $M_{(2,1)}$ & $M_{(4,1)}$ & $M_{(1,2)}$ & $M_{(3,2)}$ & $M_{(0,3)}$ & $M_{(2,3)}$ & $M_{(4,3)}$ & ${}^1_5 M_{(2,4)}$ & ${}^1_5 M_{(0,4)}$ \\ \hline
step 1 &  &  &  &  &  &  &  & ${}^1_5 M_{(0,4)}$ & ${}^1_5 M_{(2,4)}$ & $M_{(4,4)}=0$ & $\Omega^{-1}$ & $\Omega^{-1}$ \\ \hline
$\A_1$ & $M_{(1,0)}$ & $M_{(3,0)}$ & $M_{(0,1)}$ & $M_{(2,1)}$ & $M_{(4,1)}$ & $M_{(1,2)}$ & $M_{(3,2)}$ & ${}^1_1M_{(0,0)}$ & ${}^1_1 M_{(2,0)}$ & ${}^1_1 M_{(4,0)}$ & ${}^1_{10} M_{(1,4)}$ & ${}^1_{10} M_{(3,4)}$ \\ \hline
step 2 &  &  &  &  &  & ${}^1_{10} M_{(1,4)}$ & ${}^1_{10} M_{(3,4)}$ &  & & & $\Omega^{-1}$ & $\Omega^{-1}$ \\ \hline
$\A_2$ & $M_{(1,0)}$ & $M_{(3,0)}$ & $M_{(0,1)}$ & $M_{(2,1)}$ & $M_{(4,1)}$ & ${}^1_1 M_{(1,1)}$ & ${}^1_1 M_{(3,1)}$ & ${}^1_1 M_{(0,0)}$ & ${}^1_1 M_{(2,0)}$ & ${}^1_1 M_{(4,0)}$ & ${}^1_{15} M_{(2,4)}$ & ${}^1_{15} M_{(0,4)}$ \\ \hline
step 3 &  &  & ${}^1_{15} M_{(0,4)}$ & ${}^1_{15} M_{(2,4)}$ & $M_{(4,4)}=0$ &  &  &  &  &  & $\Omega^{-1}$ & $\Omega^{-1}$ \\ \hline
$\A_3$ & $M_{(1,0)}$ & $M_{(3,0)}$ & $M_{(0,2)}$ & ${}^1_1 M_{(2,2)}$ & ${}^1_1 M_{(4,2)}$ & ${}^1_1 M_{(1,1)}$ & ${}^1_1 M_{(3,1)}$ & ${}^1_1 M_{(0,0)}$ & ${}^1_1 M_{(2,0)}$ & ${}^1_1 M_{(4,0)}$ & ${}^1_{20} M_{(1,4)}$ & ${}^1_{20} M_{(3,4)}$ \\ \hline
step 4 & ${}^1_{20} M_{(1,4)}$ & ${}^1_{20} M_{(3,4)}$ &  &  &  &  &  &  &  &  & $\Omega^{-1}$ & $\Omega^{-1}$  \\ \hline
$\A_4$ & ${}^1_1 M_{(1,3)}$ & ${}^1_1 M_{(3,3)}$ & ${}^1_1 M_{(0,2)}$ & ${}^1_1 M_{(2,2)}$ & ${}^1_1 M_{(4,2)}$ & ${}^1_1 M_{(1,1)}$ & ${}^1_1 M_{(3,1)}$ & ${}^1_1 M_{(0,0)}$ & ${}^1_1 M_{(2,0)}$ & ${}^1_1 M_{(4,0)}$ & ${}^1_1 M_{(2,4)}$ & ${}^1_1 M_{(0,4)}$ \\ \hline
$\A_5$ & ${}_1M_{16}$ & ${}_1M_{18}$ & ${}_1M_{10}$ & ${}_1M_{12}$ & ${}_1M_{14}$ & ${}_1M_{6}$ & ${}_1M_{8}$ & ${}_1M_{0}$ & ${}_1M_{2}$ & ${}_1M_{4}$ & ${}_1M_{22}$ & ${}_1M_{20}$ \\ \hline
\end{tabular} 
\end{table} 

\pagebreak

\begin{table}
\caption{Mapping from principal block to non-principal block of $SL_2(3^3)$, local stable category equivalent.}
\centering 
\renewcommand{\arraystretch}{1.3}
\begin{tabular}{|l|c|c|c|c|c|c|c|c|c||c|c|c||c|}
\multicolumn{7}{l}{Shorthand: $U_i \Omega^j M_a$ written as ${}^j_i M_a$}& \multicolumn{7}{r}{$U$ subscripts modulo $3^3-1=26$} \\  \hline
$B_0(3^2)$ & \multicolumn{9}{|c||}{Layer 2} & \multicolumn{3}{|c||}{Layer 1} & Layer 0 \\ \hline 
& $S_0$ & $S_2$ & $S_4$ & $S_6$ & $S_8$ & $S_{10}$ & $S_{12}$ & $S_{14}$ & $S_{16}$ & $S_{18}$ & $S_{20}$ & $S_{22}$ & $S_{24}$ \\ \hline
$b_2$ & \multicolumn{5}{|c|}{$b_2=0$} & \multicolumn{4}{|c||}{$b_2=1$} & \multicolumn{4}{|c|}{$b_2=2$} \\ \hline
$b_1$ & \multicolumn{9}{|c||}{} & \multicolumn{2}{|c|}{$b_1=0$} & $b_1=1$ & $b_1=2$ \\ \hline
$\A_0=\B_0$ & $M_{0}$ & $M_{2}$ & $M_{4}$ & $M_{6}$ & $M_{8}$ & $M_{10}$ & $M_{12}$ & $M_{14}$ & $M_{16}$ & $M_{18}$ & $M_{20}$ & $M_{22}$ & ${}^1_3 M_{25}$ \\ \hline
step 1; {\footnotesize$r=0$} & &  &  &  &  &  &  &  & &  &  & ${}^1_3 M_{25}$ & $\Omega^{-1}$ \\ \hline
$\A_1$ & $M_{0}$ & $M_{2}$ & $M_{4}$ & $M_{6}$ & $M_{8}$ & $M_{10}$ & $M_{12}$ & $M_{14}$ & $M_{16}$ & $M_{18}$ & $M_{20}$ & ${}^1_9 M_{19}$ & ${}^1_6 M_{24}$ \\ \hline
step 2; {\footnotesize$r=0$} & &  &  &  &  &  & & &  & ${}^1_6 M_{24}$ & $M_{26}=0$ & & $\Omega^{-1}$ \\ \hline
$\A_2$ & $M_{0}$ & $M_{2}$ & $M_{4}$ & $M_{6}$ & $M_{8}$ & $M_{10}$ & $M_{12}$ & $M_{14}$ & $M_{16}$ & ${}^1_9 M_{21}$ & ${}^1_9 M_{23}$ & ${}^1_9 M_{19}$ & ${}^1_9 M_{25}$ \\ \hline
step 3; {\footnotesize$r=1$} & &  &  &  &  & ${}^1_9 M_{19}$ & ${}^1_9 M_{21}$ & ${}^1_9 M_{23}$ & ${}^1_9 M_{25}$ & $\Omega^{-1}$ & $\Omega^{-1}$ & $\Omega^{-1}$ & $\Omega^{-1}$ \\ \hline
$\A_3$ & $M_{0}$ & $M_{2}$ & $M_{4}$ & $M_{6}$ & $M_{8}$ & ${}^1_1 M_{1}$ & ${}^1_1 M_{3}$ & ${}^1_1 M_{5}$ & ${}^1_1 M_{7}$ & ${}_9 M_{21}$ & ${}_9 M_{23}$ & ${}_9 M_{19}$ & ${}^1_{12} M_{24}$ \\ \hline
step 4; {\footnotesize$r=0$} & &  &  &  &  &  & & &  & ${}^1_{12} M_{24}$ & $M_{26}=0$ & & $\Omega^{-1}$ \\ \hline
$\A_4$ & $M_{0}$ & $M_{2}$ & $M_{4}$ & $M_{6}$ & $M_{8}$ & ${}^1_1 M_{1}$ & ${}^1_1 M_{3}$ & ${}^1_1 M_{5}$ & ${}^1_1 M_{7}$ & ${}^1_{18} M_{18}$ & ${}^1_{18} M_{20}$ & ${}_9 M_{19}$ & ${}^1_{15} M_{25}$ \\ \hline
step 5; {\footnotesize$r=0$} & &  &  &  &  &  &  &  & &  &  & ${}^1_{15} M_{25}$ & $\Omega^{-1}$ \\ \hline
$\A_5$ & $M_{0}$ & $M_{2}$ & $M_{4}$ & $M_{6}$ & $M_{8}$ & ${}^1_1 M_{1}$ & ${}^1_1 M_{3}$ & ${}^1_1 M_{5}$ & ${}^1_1 M_{7}$ & ${}^1_{18} M_{18}$ & ${}^1_{18} M_{20}$ & ${}^1_{18} M_{22}$ & ${}^1_{18} M_{24}$ \\ \hline
step 6; {\footnotesize$r=1$} & ${}^1_{18} M_{18}$ & ${}^1_{18} M_{20}$ & ${}^1_{18} M_{22}$ &  ${}^1_{18} M_{24}$ & $M_{26}=0$ &  &  &  &  & $\Omega^{-1}$ & $\Omega^{-1}$ & $\Omega^{-1}$ & $\Omega^{-1}$ \\ \hline
$\A_6$ & ${}^1_1 M_{9}$ & ${}^1_1 M_{11}$ & ${}^1_1 M_{13}$ & ${}^1_1 M_{15}$ & ${}^1_1 M_{17}$ & ${}^1_1 M_{1}$ & ${}^1_1 M_{3}$ & ${}^1_1 M_{5}$ & ${}^1_1 M_{7}$ & ${}_{18} M_{18}$ & ${}_{18} M_{20}$ & ${}_{18} M_{22}$ & ${}^1_{21} M_{25}$ \\ \hline
step 7; {\footnotesize$r=0$} & &  &  &  &  &  &  &  & &  &  & ${}^1_{21} M_{25}$ & $\Omega^{-1}$ \\ \hline
$\A_7$ & ${}^1_1 M_{9}$ & ${}^1_1 M_{11}$ & ${}^1_1 M_{13}$ & ${}^1_1 M_{15}$ & ${}^1_1 M_{17}$ & ${}^1_1 M_{1}$ & ${}^1_1 M_{3}$ & ${}^1_1 M_{5}$ & ${}^1_1 M_{7}$ & ${}_{18} M_{18}$ & ${}_{18} M_{20}$ & ${}^1_1 M_{19}$ & ${}^1_{24} M_{24}$ \\ \hline
step 8; {\footnotesize$r=0$} & &  &  &  &  &  & & &  & ${}^1_{24} M_{24}$ & $M_{26}=0$ & & $\Omega^{-1}$ \\ \hline
$\A_8$ & ${}^1_1 M_{9}$ & ${}^1_1 M_{11}$ & ${}^1_1 M_{13}$ & ${}^1_1 M_{15}$ & ${}^1_1 M_{17}$ & ${}^1_1 M_{1}$ & ${}^1_1 M_{3}$ & ${}^1_1 M_{5}$ & ${}^1_1 M_{7}$ & ${}^1_1 M_{21}$ & ${}^1_1 M_{23}$ & ${}^1_1 M_{19}$ & ${}^1_1 M_{25}$ \\ \hline
$\A_9$ & ${}_1M_9$ & ${}_1M_{11}$ & ${}_1M_{13}$ & ${}_1M_{15}$ & ${}_1M_{17}$ & ${}_1M_{1}$ & ${}_1M_{3}$ & ${}_1M_{5}$ & ${}_1M_{7}$ & ${}_1M_{21}$ & ${}_1M_{23}$ & ${}_1M_{19}$ & ${}_1M_{25}$ \\ \hline
\end{tabular} 
\end{table}
\end{landscape}



\end{document}